\documentclass{mcom-l}
\newtheorem{theorem}{Theorem}[section]
\newtheorem{lemma}[theorem]{Lemma}
\newtheorem{corollary}[theorem]{Corollary}

\theoremstyle{definition}

\newtheorem{algorithm}[theorem]{Algorithm}

\theoremstyle{remark}
\newtheorem{remark}[theorem]{Remark}
\numberwithin{equation}{section}

\numberwithin{theorem}{section}
\numberwithin{equation}{section}

\usepackage[all]{xy}
\usepackage{amssymb,amsmath,stmaryrd}
\usepackage{amscd}
\usepackage{braket,amsfonts,amsopn}
\usepackage{bm}
\usepackage{makecell,multirow,diagbox}
\usepackage{mathrsfs}
\usepackage{graphicx,epstopdf}
\usepackage{subfigure}
\usepackage[pdftex,colorlinks]{hyperref}

\DeclareMathOperator{\divg}{div}
\DeclareMathOperator{\Divg}{Div}

\DeclareMathOperator{\curl}{curl}
\DeclareMathOperator{\rot}{rot}
\DeclareMathOperator{\Rot}{Rot}
\DeclareMathOperator{\tr}{tr}
\DeclareMathOperator{\osc}{osc}
\DeclareMathOperator{\Div}{Div}

\newcommand{\ab}[2]{\langle#1,#2\rangle}

\newcommand{\Tl}{\mathcal{T}_\ell}
\newcommand{\Sl}{\mathcal{S}_\ell}

\newcommand{\CT}{\mathcal{T}}
\newcommand{\CE}{\mathcal{E}}
\newcommand{\CS}{\mathcal{S}}

\newcommand{\VB}{\mathbb{V}}

\newcommand{\CI}{\mathcal{I}}

\newcommand{\CC}{\mathcal{C}}

\newcommand{\lr}[1]{\llbracket#1\rrbracket}
\newcommand{\vertiii}[1]{{\left\vert\kern-0.25ex\left\vert\kern-0.25ex\left\vert #1
    \right\vert\kern-0.25ex\right\vert\kern-0.25ex\right\vert}}

\def\yulAuthor{Yuwen Li}
\def\yulShortAuthor{Y. Li}

\def\yulAddress{Department of Mathematics, University of California San Diego,
 La Jolla, California 92093-0112.}
\def\yulcurrAddress{Department of Mathematics, The Pennsylvania State University, University Park, Pennsylvania 16801.}

\def\yulEmail{yuwenli925@gmail.com}

\title{Quasi-optimal adaptive mixed finite element methods for controlling natural norm errors}

\def\shortTitle{Quasi-optimal AMFEM for controlling natural norm errors}

\def\myAMS{65N12, 65N15, 65N30, 65N50, 41A25}

\def\myAbstract{
For a generalized Hodge Laplace equation, we prove the quasi-optimal convergence rate of an adaptive mixed finite element method. This adaptive method can control the error in the natural mixed variational norm {when the space of harmonic forms is trivial}. In particular, we obtain new quasi-optimal adaptive mixed methods for the Hodge Laplace,  Poisson, and Stokes equations. Comparing to existing adaptive mixed methods, the new methods control errors in both variables.
}

\begin{document}


\bibliographystyle{amsplain}
\author[\yulShortAuthor]{\yulAuthor}
\address{\yulAddress}
\curraddr{\yulcurrAddress}
\email{\yulEmail}

\subjclass[2010]{Primary \myAMS}
\date{July 08, 2019}
\begin{abstract}\myAbstract\end{abstract}
\maketitle
\markboth{
\yulShortAuthor }{\shortTitle}


\section{Introduction}\label{sec1}
Adaptive finite element method (AFEM) has been an active research area since the pioneering work \cite{BR1978}, see, e.g., \cite{Verfurth2013,BS2001,NSV2009} for a thorough introduction. Comparing to finite element methods using quasi-uniform meshes, AFEMs can achieve quasi-optimal convergence rate by producing a sequence of graded meshes  resolving singularity arising from irregular data of differential equations and domains with corners or slits. Typically, AFEMs can be described by the feedback loop
\begin{equation*}
   \begin{CD}
\textsf{SOLVE}@>>>\textsf{ESTIMATE}@>>>\textsf{MARK}@>>>\textsf{REFINE}.
\end{CD}
\end{equation*} Given a conforming mesh $\CT_\ell,$ the routine \textsf{SOLVE} returns the finite element solution $U_\ell$ of a discrete problem on $\Tl$. Based on the computed solution, \textsf{ESTIMATE} returns a collection of error indicators $\{\CE_\ell(U_\ell,T)\}_{T\in\Tl}$, which is used in \textsf{MARK} to select a subset $\mathcal{M}_\ell$ of $\Tl$. A conforming subtriangulation $\CT_{\ell+1}$ is then obtained using \textsf{REFINE} with the input $\mathcal{M}_\ell$ and \textsf{SOLVE} is called on $\CT_{\ell+1}$ again. Despite the popularity of AFEMs in practice, \cite{BV1984} for the one-dimensional boundary value problem had been the only convergence result of AFEMs for a long time. Using a bulk chasing marking strategy in \textsf{MARK}, D\"orfler \cite{Dorfler1996} first proved that the Lagrange element solution $U_\ell$ converges to the exact solution $U$ in the energy norm for Poisson's equation in $\mathbb{R}^2$ provided the initial mesh is fine enough. Readers are referred to \cite{MNS2000,BDD2004,Stevenson2007,CKNS2008,DKS2016} and references therein for further important progress in the analysis of convergence and optimality of AFEMs for symmetric and positive-definite elliptic problems. Of particular relevance in this paper is \cite{FFP2014}, where the authors used weak convergence technique to prove the quasi-optimal convergence rate of AFEMs for nonsymmetric and nonlinear elliptic problems.

The mixed finite element method (MFEM) is designed to numerically solve systems  of  partial  differential equations arising from elasticity, fluids, electromagnetism, computational geometry etc. In contrast to AFEMs based on positive-definite formulations, the difficulty in analysis of convergence and optimality of adaptive mixed finite element methods (AMFEMs) are two-fold. First, the a posteriori error analysis hinges on delicate decomposition results and possibly bounded commuting quasi-interpolations onto a sequence of finite elements spaces, see, e.g., \cite{Alonso1996,Schoberl2008,DH2014}. In addition, those quasi-interpolations are even required to locally preserve finite element functions when deriving discrete reliability, see, e.g., \cite{Demlow2017,ZCSWX2012}. Second, the exact solution $U$ of a system of equations is generally only a critical point of some variational principle. Hence $U-U_{\ell+1}$ is not orthogonal to $U_\ell-U_{\ell+1}$ and a technical quasi-orthogonality is indispensable, see, e.g., \cite{CH2006,CHX2009,BM2008,CFPP2014,Feischl2019}.

Consider the popular model problem for the analysis of MFEMs: Find $(\sigma,u)\in H(\divg;\Omega)\times L^2(\Omega)$ such that
\begin{equation}\label{mixPoisson}
    \begin{aligned}
        \ab{\sigma}{\tau}-\ab{\divg\tau}{u}&=0,\quad\tau\in H(\divg;\Omega),\\
        \ab{\divg\sigma}{v}&=\ab{f}{v},\quad v\in L^2(\Omega).
    \end{aligned}
\end{equation}
In fact \eqref{mixPoisson} is the mixed formulation of Poisson's equation.
Let $\{(\sigma_\ell,u_\ell,\Tl)\}_{\ell\geq0}$ be the finite element solutions and meshes produced by some AMFEM for \eqref{mixPoisson} using Raviart--Thomas (RT) or Brezzi--Douglas--Marini (BDM) elements, see \cite{RT1977,BDM1985}. Under mild assumptions, it has been shown in e.g., \cite{CHX2009,BM2008,HuangXu2012,HLMS2019} that $\sigma_\ell$ converges to $\sigma$ in the $L^2$-norm with quasi-optimal convergence rate. As far as we know, most existing AMFEMs for Poisson's equation are not able to control the error $\|\sigma-\sigma_\ell\|_{H(\divg)}+\|u-u_\ell\|$ because most error indicators and quasi-orthogonality in literature are not designed for the natural $H(\divg)\times L^2$-norm. This limitation seems not so severe for Poisson's equation, since $\divg(\sigma-\sigma_\ell)$ is trivially controlled by $f$ and the scalar variable $u$ is practically less important than the flux $\sigma$. However, there are still several works on a posteriori $H(\divg)\times L^2$-error estimation of mixed methods for \eqref{mixPoisson}, see, e.g., \cite{BV1996,CC1997,DH2014}. An analysis of optimality of AMFEMs for \eqref{mixPoisson} based on the $H(\divg)\times L^2$-error estimator, under the assumption that the initial triangulation $\mathcal{T}_0$ is fine enough in general, was presented in \cite{CarHella2017}.

Poisson's equation is a special case of the Hodge Laplace equation $(d\delta+\delta d)u=f$, which is the model problem in the theory of finite element exterior calculus (FEEC) developed by Arnold, Falk, and Winther \cite{AFW2006,AFW2010}. Here $d$ is the exterior derivative for differential forms and $\delta$ is the adjoint operator of $d$. In general, the Hodge Laplace equation is solved by the mixed method \eqref{DHL} in the FEEC literature. Adaptivity in FEEC has been an active research area in recent years. Using their regular decomposition and commuting quasi-interpolation, Demlow and Hirani \cite{DH2014} developed the first reliable a posteriori error estimator for controlling the error $\|\sigma-\sigma_\ell\|_V+\|p-p_\ell\|+\|u-u_\ell\|_V$ of the mixed method \eqref{DHL}. At the same time, Falk and Winther \cite{FW2014} constructed a technical local bounded commuting interpolation connecting the de Rham complex \eqref{deRham} and its finite element subcomplex. Using these ingredients, \cite{Demlow2017,CW2017,YL2019,HLMS2019} recently developed quasi-optimal AMFEMs for problems posed on the de Rham complex. For the Hodge Laplace equation, we \cite{YL2019} developed an AMFEM for controlling $\|\sigma-\sigma_\ell\|_V$ with quasi-optimal convergence rate and another AMFEM for controlling $\|\sigma-\sigma_\ell\|_V+\|p-p_\ell\|+\|d(u-u_\ell)\|$ without convergence rate. However, we are not aware of any existing AMFEM for the Hodge Laplace equation for controlling the error $\|\sigma-\sigma_\ell\|_V+\|u-u_\ell\|_V$ in the natural $V\times V$ mixed variational norm, where $V=H\Lambda^{k-1}(\Omega)$ or $H\Lambda^k(\Omega)$  is the Sobolev space of differential $(k-1)$-forms or $k$-forms, see Subsection \ref{secdeRham} for details.

On the other hand, Cai et al.~\cite{CaiWang2007,Cai2010a,Cai2010b} developed the pseudostress-velocity formulation \eqref{STOKES} of the Stokes equation, which can be numerically solved by the classical RT and BDM element mixed methods. Let $\bm{\sigma}$ denote the pseudostress, $\bm{u}$ the velocity, and $(\bm{\sigma}_\ell,\bm{u}_\ell)$ the finite element solutions produced by some AMFEM. Following the analysis of AMFEMs for Poisson's equation, \cite{CGS2013, HuYu2018} recently developed quasi-optimal AMFEMs for the pseudostress-velocity formulation that control the error $\|\CC^\frac{1}{2}(\bm{\sigma}-\bm{\sigma}_\ell)\|$, where $\CC$ is a positive semi-definite operator given in \eqref{CC}. Since $\|\CC^\frac{1}{2}\cdot\|$ is only a semi-norm, incorporation of $\|\Divg(\bm{\sigma}-\bm{\sigma}_\ell)\|$ is necessary for achieving norm convergence. Unlike Poisson's equation, the velocity field $\bm{u}$ in fluids is clearly an important physical quantity. From this perspective, an AMFEM for controlling $\|\bm{u}-\bm{u}_\ell\|$ is favorable.

Motivated by the Hodge Laplace and Stokes equations,
this paper is devoted to the quasi-optimal adaptive mixed method for controlling the \emph{natural norm} error. To this end, we consider the \emph{generalized} Hodge Laplace equation \eqref{GHL}, which covers the mixed formulation of Poisson's equation, the pseudostress-velocity formulation of the Stokes equation, and the Hodge Laplace equation with index $k\leq n-1$. {Throughout this paper, we assume the space of harmonic $k$-forms $\mathfrak{H}^k=\{0\}$ to avoid additional difficulty caused by nontrivial harmonic forms.} In Section \ref{secex}, we will restate our results in the classical context. The main contribution of this paper is as follows.
\begin{enumerate}
    \item Using the Demlow--Hirani regular decomposition and Falk--Winther cochain projection in FEEC, we obtain the first quasi-optimal AMFEM for the Hodge Laplace equation that reduces the error in the $V\times V$-norm {provided $\mathfrak{H}^k=\{0\}$}. In the case of $k=n$, $\CC=\text{id}$, i.e., Poisson's equation, we obtain an AMFEM that reduces the error in the $H(\divg)\times L^2$-norm.

    \item By posing the Stokes equation on the de Rham complex of \emph{vector-valued} differential forms, we modify the aforementioned tools in FEEC to derive a reliable and efficient a posteriori error estimator, and the first  AMFEM for the Stokes equation that reduces the error $\|\CC^\frac{1}{2}(\bm{\sigma}-\bm{\sigma}_\ell)\|+\|\Divg(\bm{\sigma}-\bm{\sigma}_\ell)\|+\|\bm{u}-\bm{u}_\ell\|$. The authors in \cite{CKP2011} used $\bm{u}_\ell$ to compute a more accurate postprocessed approximation $\bm{u}^*_\ell$ and derived an error estimator for $\|\bm{u}-\bm{u}_\ell^*\|.$ However, the reliability of such estimator depends on the $H^2$-regularity of $\Omega$, e.g., $\Omega$ is convex.

    \item Our results for the Poisson and Stokes equations hold on general Lipschitz polyhedral domain $\Omega$. In contrast, existing analysis of AMFEMs in e.g., \cite{CHX2009,BM2008,CGS2013,HuYu2018} assumes that the Helmholtz decomposition contains no harmonic vector fields, which hinges on the topology of the domain $\Omega$, e.g., the $(n-1)$-th Betti number of $\Omega$ is $0$.
\end{enumerate}

An important and novel ingredient of our convergence analysis is the quasi-orthogonality in Theorem \ref{relations} and Equation \eqref{trueqo}. We observe that the $H^1$-regular decomposition in \cite{DH2014} yields compact operators $\bar{\mathcal{K}}^k_1, \bar{\mathcal{K}}^k_2$ in Corollary \ref{disrd} and develop a weak convergence result in Theorem \ref{weakPG} for the Petrov--Galerkin method. Note that $\bar{\mathcal{K}}^k_1, \bar{\mathcal{K}}^k_2$ map weakly convergent sequences to strongly convergent ones. Using this fact and the $L^2$-bounded smoothed projection in \cite{CW2008}, we obtain the quasi-orthogonality between $u-u_{\ell+1}$ and $u_\ell-u_{\ell+1}$. Combining it with the quasi-orthogonality between $\sigma-\sigma_{\ell+1}$ and $\sigma_\ell-\sigma_{\ell+1}$ obtained in \cite{YL2019}, the quasi-optimal convergence rate follows with a somehow standard procedure using the idea of estimator reduction, see \cite{FFP2014}. Feischl et al.~first used the weak convergence  technique to prove quasi-optimal convergence rate of AFEMs in \cite{FFP2014}, where they observed that the lower order terms in $2^{nd}$ order elliptic equation are \emph{compact} perturbations. As far as we know, there is no convergence analysis of adaptive mixed methods in literature based on the weak convergence technique. For the mixed Poisson's problem \eqref{mixPoisson}, the analysis of the $H(\divg)\times L^2$-based AMFEM in \cite{CarHella2017} utilizes the concept of generalized quasi-orthogonality and a superconvergence estimate with fineness assumption on the initial mesh. As we shall show in Section \ref{secqo}, for the simple but important problem \eqref{mixPoisson}, our novel analysis of quasi-orthogonality does not rely on such assumption on the initial mesh.

The rest of this paper is organized as follows. In Section \ref{secHilbert}, we introduce the closed Hilbert complex, de Rham complex, and the generalized Hodge Laplace equation. In Section \ref{secEstimate}, we derive reliable and efficient a posteriori error estimator for the generalized Hodge Laplace equation on the de Rham complex. Section \ref{secqo} is devoted to the convergence and optimality analysis of the algorithm \textsf{AMFEM} and two modified adaptive algorithms. In Section \ref{secex}, we use previous results and correspondence between functions and differential forms to obtain results on scalar Poisson, vector Poisson, and Stokes equations.

\section{Hilbert complex and de Rham complex}\label{secHilbert}
Following the convention of \cite{AFW2006,AFW2010}, we  introduce FEEC in this section.
\subsection{Hilbert complex and approximation}
Given Hilbert spaces $X_1, X_2$,  we say $T: X_1\rightarrow X_2$ is a closed, densely-defined  operator if the domain $D(T)=\{v\in X_1: Tv\in X_2\text{ is defined}\}$ is a dense subspace of $X_1$, $T: D(T)\rightarrow X_2$ is linear, and the graph $\{(v,Tv): v\in D(T)\}$ is a closed subset of $X_1\times X_2$. Let $\ab{\cdot}{\cdot}_{X_i}$ denote the inner product on $X_i$, $i=1,2$.
The adjoint operator $T^*: X_2\rightarrow X_1$ is defined to be the operator whose domain is
$$D(T^*)=\{v\in X_2: \exists w\in X_1,\text{ such that }\ab{Tu}{v}_{X_2}=\ab{u}{w}_{X_1}\text{ for all }u\in D(T)\},$$
in which case $T^*v:=w.$
The adjoint $T^*$ is also a densely-defined, closed operator.  Let $R(T)$ be the range of $T$, $N(T)$ the kernel of $T$, the closed range theorem holds:
\begin{align}\label{crange}
    R(T)^\perp=N(T^*),
\end{align}
where $\perp$ denotes the orthogonal complement operation.

Consider the closed Hilbert complex $(W,d):$
\begin{equation*}
\cdots\rightarrow{} W^{k-1}\xrightarrow{d^{k-1}}W^k\xrightarrow{d^k}W^{k+1}\xrightarrow{d^{k+1}}\cdots,
\end{equation*}
i.e., for each index $k$, $W^{k}$ is a Hilbert space equipped with the inner product $\ab{\cdot}{\cdot}$ and norm $\|\cdot\|$, $d^{k}: W^{k}\rightarrow W^{k+1}$ is a densely-defined,  closed operator, $R(d^{k})\subseteq D(d^{k+1})$ is closed in $W^{k+1}$, and $d^{k+1}\circ d^{k}=0$. Let $\mathfrak{Z}^{k}=N(d^{k})$, $\mathfrak{B}^{k}=R(d^{k-1})$, and  $\mathfrak{H}^{k}=\mathfrak{Z}^{k}\cap\mathfrak{B}^{k\perp}$ denote the space of abstract harmonic $k$-forms. $\mathfrak{H}^{k}$ is also called the $k$-th cohomology group since $\mathfrak{H}^{k}\cong\mathfrak{Z}^k/\mathfrak{B}^k$.
The cochain complex $(W,d)$ has the domain complex $(V,d)$ as a subcomplex:
\begin{equation*}
\cdots\rightarrow V^{k-1}\xrightarrow{d^{k-1}}V^{k}\xrightarrow{d^{k}}V^{k+1}\xrightarrow{d^{k+1}}\cdots.
\end{equation*}
Here $V^{k}=D(d^k)$ is the domain of $d^{k}$, equipped with the $V$-inner product
\begin{align*}
\ab{u}{v}_V:=\ab{u}{v}+\ab{d^ku}{d^kv},
\end{align*}
and corresponding $V$-norm $\|\cdot\|_V.$
Let $\mathfrak{Z}^{k\perp V}=\{v\in V^k: \ab{v}{z}=0\text{ for all }z\in\mathfrak{Z}^k\}.$ There exists a constant $c_P>0,$ such that
\begin{equation}\label{Poincare}
    \|v\|_V\leq c_P\|d^kv\|\text{ for all }v\in\mathfrak{Z}^{k\perp V}.
\end{equation}
In FEEC literature, \eqref{Poincare} is called the Poincar\'e inequality.

For each index $k$, choose a finite-dimensional subspace $V_\ell^k$ of $V^k$. We assume that $dV_\ell^k\subseteq V_\ell^{k+1}$ so that $(V_\ell,d)$ is a subcomplex of $(V,d)$. Let $W_\ell^k$ be the same space $V_\ell^k$ but equipped with the $W$-inner product $\ab{\cdot}{\cdot}$.
Similarly to the continuous case, let $\mathfrak{Z}^k_\ell=N(d^k_\ell),\ \mathfrak{B}^k_\ell=R(d^{k-1}_{\ell}),$ and $\mathfrak{H}^k_\ell=\mathfrak{B}_\ell^k\cap\mathfrak{Z}_\ell^{k\perp}$. Note that in general $\mathfrak{Z}_\ell^{k\perp}\not\subseteq\mathfrak{Z}^{k\perp}$ and $\mathfrak{H}_\ell^{k}\not\subseteq\mathfrak{H}^{k}$. In order to derive a posteriori error estimate on $(V_\ell,d)$, we assume the existence of a bounded cochain projection $\pi_\ell$ from $(V,d)$ to $(V_\ell,d)$. To be precise, for each index $k$, $\pi_\ell^{k}$ maps $V^{k}$ onto $V_\ell^{k}$, $\pi_\ell^{k}|_{V_\ell^{k}}=$id, $d^{k}\pi_\ell^{k}=\pi_\ell^{k+1}d^{k}$, and $\|\pi_\ell^{k}\|_{V}=\|\pi_\ell^{k}\|_{V^{k}\rightarrow V_\ell^{k}}<\infty$ is uniformly bounded with respect to the discretization parameter $\ell$. It has been shown in \cite{AFW2010} that the discrete Poincar\'e inequality holds:
\begin{equation*}
    \|v\|_V\leq c_P\|\pi^k_\ell\|_V\|d^kv\|\text{ for all }v\in\mathfrak{Z}_\ell^{k\perp}.
\end{equation*}

\subsection{Generalized Hodge Laplacian and approximation} For each index $k$, let $d_k^*$ denote the adjoint operator of $d^{k-1}: W^{k-1}\rightarrow W^{k}$ and $V_k^*=D(d_k^*)$. Throughout the rest of this paper, we may drop the superscript or subscript $k$ provided no confusion arises. On the closed Hilbert complex $(W,d)$,
Arnold, Falk and Winther \cite{AFW2010} considered the abstract Hodge Laplace equation
$$(dd^*+d^*d)u=f,$$
where $f\in W^k$ and $u\in V^k\cap V_k^*$  satisfy the compatibility conditions $f\perp\mathfrak{H}^k$, $u\perp\mathfrak{H}^k.$
Note that $u\in V^k\cap V_k^*$ and it is difficult to construct appropriate finite element subspaces of $V^k\cap V_k^*$. Therefore, the authors in \cite{AFW2010} considered the mixed formulation: Find $(\sigma,u,q)\in V^{k-1}\times V^k\times\mathfrak{H}^k$, such that
\begin{equation}\label{HL}
\begin{aligned}
\ab{\sigma}{\tau}-\ab{d\tau}{u}&=0,\quad&&\tau\in V^{k-1},\\
\ab{d\sigma}{v}+\ab{du}{dv}+\ab{v}{p}&=\ab{f}{v},&& v\in V^{k},\\
\ab{u}{q}&=0,&& q\in\mathfrak{H}^k.
\end{aligned}
\end{equation}
Using the discrete complex $(V_\ell,d)$, the mixed method for \eqref{HL} seeks $(\sigma_\ell,u_\ell,p_\ell)\in$ $V_\ell^{k-1}\times V_\ell^{k}\times\mathfrak{H}_\ell^k$ such that
\begin{equation}\label{DHL}
\begin{aligned}
\ab{\sigma_\ell}{\tau}-\ab{d\tau}{u_\ell}&=0,\quad&&\tau\in V_\ell^{k-1},\\
\ab{d\sigma_\ell}{v}+\ab{du_\ell}{dv}+\ab{v}{p_\ell}&=\ab{f}{v},&& v\in V_\ell^{k},\\
\ab{u_\ell}{q}&=0,&& q\in\mathfrak{H}_\ell^k.
\end{aligned}
\end{equation}
In general, {$\mathfrak{H}_\ell^{k}$ is not a subspace of $\mathfrak{H}^{k}$}. For the sake of simplicity, we assume $\mathfrak{H}^k=\{0\}$ and consider the generalized Hodge Laplacian problem: Find $(\sigma,u)\in V^{k-1}\times V^k$, such that
\begin{equation}\label{GHL}
\begin{aligned}
\ab{\CC\sigma}{\tau}-\ab{d\tau}{u}&=0,\quad&&\tau\in V^{k-1},\\
\ab{d\sigma}{v}+\ab{du}{dv}&=\ab{f}{v},&& v\in V^{k},
\end{aligned}
\end{equation}
where $\CC: W^{k-1}\rightarrow W^{k-1}$ is a self-adjoint, positive semi-definite, continuous linear operator. The operator $\CC$ is not necessarily positive definite with respect to $\ab{\cdot}{\cdot}$ and $\|\CC^\frac{1}{2}\cdot\|=\ab{\CC\cdot}{\cdot}^\frac{1}{2}$ may not define a norm on $W^{k-1}$. We assume that there exists a constant $C_\CC>0$ with
\begin{equation}\label{equivCC}
    C_\CC^{-1}\|\tau\|_V^2\leq\ab{\CC\tau}{\tau}+\|d\tau\|^2\leq C_\CC\|\tau\|^2_V
\end{equation}
for all $\tau\in V^{k-1}$. For $\tau_1,\tau_2\in V^{k-1}$, let $$\ab{\tau_1}{\tau_2}_{V_\CC}:=\ab{\CC\tau_1}{\tau_2}+\ab{d\tau_1}{d\tau_2}.$$
The assumption \eqref{equivCC} shows that
$\ab{\cdot}{\cdot}_{V_\CC}$ is an inner product on $V^{k-1}$ and the $V_\CC$-norm $\|\cdot\|_{V_\CC}=\ab{\cdot}{\cdot}_{V_\CC}^\frac{1}{2}$ is equivalent to $\|\cdot\|_V.$
The mixed method for solving \eqref{GHL} is to find $(\sigma_\ell,u_\ell)\in V_\ell^{k-1}\times V_\ell^k$ satisfying
\begin{subequations}\label{DGHL}
\begin{align}
\ab{\CC\sigma_\ell}{\tau}-\ab{d\tau}{u_\ell}&=0,\quad\tau\in V_\ell^{k-1},\label{DGHL1}\\
\ab{d\sigma}{v}+\ab{du_\ell}{dv}&=\ab{f}{v},\quad v\in V_\ell^{k}.\label{DGHL2}
\end{align}
\end{subequations}
Thanks to the cochain projection $\pi_\ell,$ we obtain $\mathfrak{H}_\ell^k=\pi_\ell^k(\mathfrak{H}^k)=\{0\}$ and the well-posedness of \eqref{GHL} and \eqref{DGHL}, see Theorem \ref{infsup}. Assuming $V_\ell^{k-1}\subseteq V_{\ell+1}^{k-1}, V_\ell^{k}\subseteq V_{\ell+1}^k$ and using \eqref{DGHL}, we obtain the Galerkin orthogonality
\begin{subequations}\label{error}
\begin{align}
\ab{\CC({\sigma}_{\ell+1}-\sigma_\ell)}{\tau}-\ab{d\tau}{{u}_{\ell+1}-u_\ell}&=0,\quad\tau\in V_\ell^{k-1},\label{error1}\\
\ab{d({\sigma}_{\ell+1}-\sigma_\ell)}{v}+\ab{d({u}_{\ell+1}-u_\ell)}{dv}&=0,\quad v\in V_\ell^{k}.\label{error2}
\end{align}
\end{subequations}

Let $B(\sigma,u;\tau,v)=\ab{\CC\sigma}{\tau}-\ab{d\tau}{u}+\ab{d\sigma}{v}+\ab{du}{dv}$. The next theorem shows that $B$ satisfies the continuous and discrete inf-sup condition, which implies the well-posedness of \eqref{GHL} and \eqref{DGHL}. The proof is the same as Theorem 3.2 in \cite{AFW2010}.
\begin{theorem}\label{infsup}
Assume $\mathfrak{H}^k=\{0\}$. There exists a constant $C_{\emph{infsup}}>0$ depending only on $c_P, C_\CC$, such that
\begin{equation*}
    C_{\emph{infsup}}\big(\|\xi\|_{V_\CC}+\|w\|_V\big)\leq\sup_{\tau\in V^{k-1},v\in V^k} \frac{B(\xi,w;\tau,v)}{\|\tau\|_{V_\CC}+\|v\|_V},
\end{equation*}
for all $\xi\in V^{k-1}, w\in V^k$.
In addition, there exists a constant $C_{\emph{infsup},\ell}>0$ depending only on $c_{P}, C_\CC, \|\pi_\ell\|_V$, such that
\begin{equation*}
C_{\emph{infsup},\ell}\big(\|\xi_\ell\|_{V_\CC}+\|w_\ell\|_V\big)\leq\sup_{\tau\in V_\ell^{k-1},v\in V_\ell^k} \frac{B(\xi_\ell,w_\ell;\tau,v)}{\|\tau\|_{V_\CC}+\|v\|_V},
\end{equation*}
for all $\xi_\ell\in V_\ell^{k-1}, w_\ell\in V_\ell^k$.
\end{theorem}

Demlow and Hirani \cite{DH2014} used the continuous inf-sup condition to derive their  error estimator for the method \eqref{DHL}. Using the discrete inf-sup condition, we obtain the discrete upper bound of the abstract natural norm error.
\begin{lemma}\label{absigmau}
For $1\leq k\leq n,$ it holds that
\begin{equation*}
\begin{aligned}
    &C_{\emph{infsup},\ell+1}\big(\|{\sigma}_{\ell+1}-\sigma_\ell\|_{V_\CC}+\|{u}_{\ell+1}-u_{\ell}\|_V\big)\\
    &\quad\leq\sup_{\tau\in{V}_{\ell+1}^{k-1}, \|\tau\|_{V_\CC}=1}\{\ab{\CC\sigma_\ell}{\tau-\pi_\ell\tau}-\ab{d(\tau-\pi_\ell\tau)}{u_\ell}\}\\
    &\qquad+\sup_{v\in{V}_{\ell+1}^{k}, \|v\|_V=1}\{\ab{f-d\sigma_\ell}{v-\pi_\ell v}-\ab{du_\ell}{d(v-\pi_\ell v)}\}.
\end{aligned}
\end{equation*}
\end{lemma}
\begin{proof}
Let $\tau\in V_{\ell+1}^{k-1}$ and $v\in V_{\ell+1}^k$. It follows from \eqref{error} and \eqref{DGHL} that
\begin{equation*}
\begin{aligned}
    &B(\sigma_{\ell+1}-\sigma_\ell,u_{\ell+1}-u_\ell;\tau,v)\\
    &\quad=B(\sigma_{\ell+1}-\sigma_\ell,u_{\ell+1}-u_\ell;\tau-\pi_\ell\tau,v-\pi_\ell v)\\
    &\quad=\ab{f}{v-\pi_\ell v}-B(\sigma_\ell,u_\ell;\tau-\pi_\ell\tau,v-\pi_\ell v).
\end{aligned}
\end{equation*}
Combining it with the discrete inf-sup condition in Theorem \ref{infsup} completes the proof.
\end{proof}

\subsection{De Rham complex and approximation}\label{secdeRham}
The de Rham complex is a canonical example of the closed Hilbert complex. Let $\Omega\subset\mathbb{R}^{n}$ be a bounded Lipschitz domain. For index $0\leq k\leq n$, let $\Lambda^{k}(\Omega)$ denote the space of all smooth $k$-forms $\omega$ which can be uniquely written as
$$\omega=\sum_{1\leq \alpha_{1}<\cdots<\alpha_{k}\leq n}\omega_{\alpha}dx^{\alpha_{1}}\wedge\cdots\wedge dx^{\alpha_{k}},$$
where each $\alpha=(\alpha_1,\ldots,\alpha_k)\in\mathbb{N}^k$ is a multi-index, each coefficient $\omega_\alpha\in C^{\infty}({\Omega})$ and $\wedge$ is the wedge product. For $\eta\in\Lambda^k(\Omega)$ with $\eta=\sum_{1\leq \alpha_{1}<\cdots<\alpha_{k}\leq n}\eta_{\alpha}dx^{\alpha_{1}}\wedge\cdots\wedge dx^{\alpha_{k}},$ the inner product of $\omega$ and $\eta$ is defined as
$$\ab{\omega}{\eta}:=\sum_{1\leq \alpha_{1}<\cdots<\alpha_{k}\leq n}\int_\Omega\omega_\alpha\eta_\alpha dx.$$ The exterior derivative $d=d^k: \Lambda^{k}(\Omega)\rightarrow\Lambda^{k+1}(\Omega)$ is given by
\begin{equation}\label{derivative}
    d\omega=\sum_{1\leq \alpha_{1}<\cdots<\alpha_{k}\leq n}\sum_{j=1}^n\frac{\partial\omega_{\alpha}}{\partial x_j}dx^j\wedge dx^{\alpha_{1}}\wedge\cdots\wedge dx^{\alpha_{k}}.
\end{equation}
Let $\|\cdot\|$ denote the $L^2$-norm given by $\ab{\cdot}{\cdot}$ and $L^{2}\Lambda^{k}(\Omega)$ the space of $k$-forms with $L^2$-coefficients. Then $d$ can be understood in the distributional sense. Let $D(d^k):=H\Lambda^{k}(\Omega)=\{\omega\in L^{2}\Lambda^{k}(\Omega): d\omega\in L^{2}\Lambda^{k+1}(\Omega)\}$. The following cochain complex
\begin{equation*}
L^{2}\Lambda^{0}(\Omega)\xrightarrow{d} L^{2}\Lambda^{1}(\Omega)\xrightarrow{d}\cdots\xrightarrow{d} L^{2}\Lambda^{n-1}(\Omega)\xrightarrow{d} L^{2}\Lambda^{n}(\Omega)
\end{equation*}
is an example of the closed Hilbert complex $(W,d)$. The $L^2$-de Rham complex [corresponds to ($V,d$)] is
\begin{equation}\label{deRham}
H\Lambda^{0}(\Omega)\xrightarrow{d} H\Lambda^{1}(\Omega)\xrightarrow{d}\cdots\xrightarrow{d} H\Lambda^{n-1}(\Omega)\xrightarrow{d} H\Lambda^{n}(\Omega).
\end{equation}
In order to characterize the adjoint of $d$, we need the Hodge star operator $\star: L^2\Lambda^{k}(\Omega)\rightarrow L^2\Lambda^{n-k}(\Omega)$ determined by
$\int_{\Omega}\omega\wedge\mu=\ab{\star\omega}{\mu}$ for all $\mu\in L^2\Lambda^{n-k}(\Omega).$ The coderivative $\delta: \Lambda^{k}(\Omega)\rightarrow\Lambda^{k-1}(\Omega)$ is then determined by $\star\delta\omega=(-1)^{k}d\star\omega$. It is well known that $d$ and $\delta$ are related by the integrating by parts formula
\begin{equation}\label{IP}
\ab{d\omega}{\mu}=\ab{\omega}{\delta\mu}+\int_{\partial\Omega}\tr\omega\wedge\tr\star\mu,\quad\omega\in\Lambda^{k}(\Omega),\quad\mu\in\Lambda^{k+1}(\Omega),
\end{equation}
where the trace operator $\tr$ on $\partial\Omega$ is the pullback for differential forms induced by the inclusion $\partial\Omega\hookrightarrow\overline{\Omega}$.
If $\Omega$ is replaced by $T$ in \eqref{IP}, then $\tr$ denotes the trace on $\partial T$ by abuse of notation. We make use of the spaces  $\mathring{H}\Lambda^k(\Omega)=\{\omega\in H\Lambda^k(\Omega): \tr\omega=0\text{ on }\partial\Omega\}$ and $\mathring{H}^*\Lambda^k(\Omega)=\star\mathring{H}\Lambda^{n-k}(\Omega)$. The next lemma characterizes the adjoint operator of $d$.
\begin{theorem}[Theorem 4.1 in \cite{AFW2010}]\label{adjoint}
Let $d$ be the exterior derivative viewed as a densely-defined, closed operator $d: L^2\Lambda^{k-1}(\Omega)\rightarrow L^2\Lambda^k(\Omega)$ with domain $H\Lambda^{k-1}(\Omega)$. Then the adjoint $d^*,$ as a densely-defined, closed operator $L^2\Lambda^k(\Omega)\rightarrow L^2\Lambda^{k-1}(\Omega)$, has domain
$\mathring{H}^{*}\Lambda^{k}(\Omega):=\star \mathring{H}\Lambda^{n-k}(\Omega)=\{\omega\in H^*\Lambda^{k}(\Omega): \tr\star
\omega=0\text{ on }\partial\Omega\}$,
and coincides with $\delta.$
\end{theorem}
The generalized Hodge Laplaican problem \eqref{GHL} on the de Rham complex uses $V^{k-1}=H\Lambda^{k-1}(\Omega), V^k=H\Lambda^k(\Omega)$. A more compact form of \eqref{GHL} is
\begin{equation}\label{compactHL}
\begin{aligned}
\CC\sigma=\delta u,\quad d\sigma+\delta d u=f.
\end{aligned}
\end{equation}
Since $u\in D(\delta)=\mathring{H}^*\Lambda^k(\Omega)$ and $du\in D(\delta)=\mathring{H}^*\Lambda^{k+1}(\Omega)$, the boundary conditions $\tr\star u=0$, $\tr\star du=0$ on $\partial\Omega$ are implicitly imposed in \eqref{compactHL}. When $\CC=\text{id}$, \eqref{compactHL} reduces to the standard Hodge Laplacian problem
\begin{equation}\label{standardHL}
\begin{aligned}
    (d\delta+\delta d)u&=f\text{ in }\Omega,\\
    \tr\star u=0,\quad\tr\star du&=0\text{ on }\partial\Omega.
\end{aligned}
\end{equation}

Throughout the rest of this paper,
$\CT_0\leq\CT_1\leq\cdots\leq\Tl\leq\cdots$ denotes a sequence of nested conforming simplicial triangulations of $\Omega$, where  $\CT_\ell\leq\CT_{\ell+1}$ means $\CT_{\ell+1}$ is a refinement of $\Tl$. For $T\in\Tl$, let $|T|$ denote the volume of $T$ and
$h_T:=|T|^{\frac{1}{n}}$. We assume that $\{\CT_\ell\}_{\ell\geq0}$ is shape regular, namely, $$\sup_{\ell\geq0}\max_{T\in\CT_\ell}\frac{r_{T}}{\rho_{T}}\leq C_{\CT_0}<\infty,$$ where $r_T$ and $\rho_{T}$ are radii of circumscribed and inscribed spheres of the simplex $T$, respectively. Let $\mathcal{P}_r\Lambda^{k}(T)$ denote the space of $k$-forms on $T$ with polynomial coefficients of degree $\leq r$. Let
\begin{align*}
    &\mathcal{P}_r\Lambda^{k}(\Tl)=\{v\in H\Lambda^k(\Omega): v|_T\in\mathcal{P}_r\Lambda^{k}(T)\text{ for all }T\in\Tl\},\\
    &\mathcal{P}^-_{r+1}\Lambda^{k}(\Tl)=\mathcal{P}_r\Lambda^{k}(\Tl)+\kappa\mathcal{P}_{r}\Lambda^{k+1}(\Tl),
\end{align*}
where $\kappa: L^2\Lambda^k(\Omega)\rightarrow L^2\Lambda^{k-1}(\Omega)$ is the interior product given by
\begin{equation}\label{kappa}
    (\kappa\omega)_x(v_1,\ldots,v_{k-1})=\omega_x(X(x),v_1,\ldots,v_{k-1}),\quad \forall v_1,\ldots,v_{k-1}\in\mathbb{R}^n,
\end{equation}
with $X(x)=(x_1,\ldots,x_n)^T$ being the position vector field. For $r\geq0,$ let
\begin{equation}\label{pairs}
    \begin{aligned}
    V_\ell^{k-1}&=\mathcal{P}_{r+1}\Lambda^{k-1}(\Tl)\text{ or }\mathcal{P}_{r+1}^-\Lambda^{k-1}(\Tl),\\
    V_\ell^k&=\mathcal{P}_{r+1}^-\Lambda^k(\Tl)\text{ or }\mathcal{P}_r\Lambda^{k}(\Tl),
\end{aligned}
\end{equation}
Other spaces $V_\ell^j$ with $j\neq k, k-1$  are chosen in the same way. In $\mathbb{R}^n$, there are $2^{n-1}$ different  discrete subcomplexes $(V_\ell,d)$ on a simplicial triangulation $\Tl$.

\section{A posteriori error estimate}\label{secEstimate}
In the rest of this paper, $A\lesssim B$ provided $A\leq C\cdot B$ and $C$ is a generic constant depending only on $C_\CC, {\CT_0}$ and $\Omega$.  Let $H^{s}\Lambda^{k}(\Omega)$ be the space of $k$-forms whose coefficients are in $H^{s}(\Omega)$. Formula \eqref{IP} still holds for $\omega\in H^{1}\Lambda^{k}(\Omega)$ and $\mu\in H^{1}\Lambda^{k+1}(\Omega).$ Let $\|\cdot\|_{H^1(\Omega)}$ denote the $H^1\Lambda^l(\Omega)$ norm with some $l$ depending on the context. Let $\ab{\cdot}{\cdot}_T$ denote the $L^2$-inner product restricted to $T$, $\|\cdot\|_{T}$ and $\|\cdot\|_{\partial T}$ the $L^{2}$-norms restricted to $T$ and $\partial T$, respectively.

To derive discrete upper bounds on the de Rham complex, we need the local $V$-bounded cochain projection $\pi_\ell$ developed by Falk and Winther \cite{FW2014}. The existence of $\pi_\ell$ implies the discrete inf-sup condition by Theorem \ref{infsup}.
\begin{theorem}[local $V$-bounded cochain projection]\label{cochainproj}
For each index $0\leq k\leq n,$ there exists a projection $\pi_\ell^k:  H\Lambda^k(\Omega)\rightarrow V_\ell^k$ commuting with the exterior derivative, i.e.~$\pi_\ell^k|_{V_\ell^k}=\emph{id}$ and $d^k\pi_\ell^k=\pi_\ell^{k+1}d^k$. For $T\in\Tl$, let $D_T=\{T^\prime\in\Tl: T^\prime\cap T\neq\emptyset\}$  and $V_{\ell}^k|_{D_T}=\{v|_{D_T}: v\in V_\ell^k\}$. Then
\begin{subequations}
\begin{align}
v-\pi_\ell^k v&=0\text{ on }T\text{ for }v\in V_{\ell}^k|_{D_T},\label{local}\\
\|\pi_\ell^kv\|_{H\Lambda^k(T)}&\lesssim\|v\|_{H\Lambda^k(D_T)}\text{ for }v\in H\Lambda^k(D_T)\label{Vbound}.
\end{align}
\end{subequations}
In adddition, for $v\in H^1\Lambda^k(D_T),$
\begin{equation}\label{approxpi}
|\pi_\ell^kv|_{H^1(T)}+h_T^{-1}\|v-\pi^k_\ell v\|_T+h_T^{-\frac{1}{2}}\|\tr(v-\pi^k_\ell v)\|_{\partial T}\lesssim|v|_{H^1(D_T)}.
\end{equation}
\end{theorem}
\begin{proof}
Properties \eqref{local} and \eqref{Vbound} are given by Falk and Winther in \cite{FW2014}.

Let $\CI_\ell: H^1\Lambda^k(\Omega)\rightarrow V_\ell^k$ be the interpolation given by applying the Scott--Zhang interpolation (see~\cite{SZ1990}) on $\Tl$ to each coefficient of the $k$-form in $H^1\Lambda^k(\Omega)$.
It follows from the property \eqref{local} and the same property of $\CI_\ell$ that
\begin{equation*}
     \pi_\ell^kv-\CI_\ell v=0\text{ on }T\text{ for }v\in V_\ell^k|_{D_T},
\end{equation*}
and thus
\begin{equation}\label{perserv}
     \|\pi_\ell^kv-\CI_\ell v\|_T\lesssim h_T|v|_{H^1(D_T)}
\end{equation}
by the Bramble--Hilbert lemma. In addition, it is well-known that
\begin{align}\label{approxI}
    |\CI_\ell v|_{H^1(T)}&\lesssim|v|_{H^1(D_T)},\quad \|v-\CI_\ell v\|_{T}\lesssim h_T|v|_{H^1(D_T)}.
\end{align}
Then using the triangle inequality, \eqref{perserv} and \eqref{approxI}, we obtain
\begin{align*}
    |\pi^k_\ell v|_{H^1(T)}\leq|\pi^k_\ell v-\CI_\ell v|_{H^1(T)}+|\CI_\ell v|_{H^1(T)}&\lesssim|v|_{H^1(D_T)},\\
    \|v-\pi^k_\ell v\|_T\leq\|v-\CI_\ell v\|_T+\|\CI_\ell v-\pi^k_\ell v\|_T&\lesssim h_T|v|_{H^1(D_T)}.
\end{align*}
Combining it with the trace inequality verifies $h_T^{-\frac{1}{2}}\|\tr(v-\pi_\ell^kv)\|_{\partial T}\lesssim|v|_{H^1(D_T)}.$
\end{proof}
Using \eqref{approxpi} and the bounded overlapping property of $D_T$, we obtain
\begin{align}\label{Gapproxpi}
    |\pi_\ell v|^2_{H^1(\Omega)}+\sum_{T\in\Tl}h_T^{-2}\|v-\pi_\ell v\|^2_T+h_T^{-1}\|\tr(v-\pi_\ell v)\|^2_{\partial T}\lesssim|v|^2_{H^1(\Omega)}.
\end{align}

In addition to $\pi_\ell$, we need an $H^1$-regular decomposition result, see Lemma 5 in \cite{DH2014}. The proof therein hinges on the technical $H^1$-solution regularity of the equation $d\varphi=g$ under the Dirichlet boundary condition, see e.g., \cite{MMM2008,Schwarz1995}. In our convergence analysis, the linearity of such regular decomposition is also required. Since only the natural boundary condition is considered, we give a simple proof of the regular decomposition below. For convenience, let $H\Lambda^{-1}(\Omega)=\{0\}.$
\begin{theorem}[regular decomposition]\label{rd}
Let $\Omega$ be a bounded Lipschitz domain in $\mathbb{R}^n.$ For $0\leq k\leq n$, there exist bounded linear operators $\mathcal{K}^k_1: H\Lambda^k(\Omega)\rightarrow H^1\Lambda^{k-1}(\Omega)$ and $\mathcal{K}^k_2: H\Lambda^k(\Omega)\rightarrow H^1\Lambda^{k}(\Omega)$, such that for $v\in H\Lambda^k(\Omega)$,
\begin{equation*}
    v=d\mathcal{K}^k_1v+\mathcal{K}^k_2v.
\end{equation*}
\end{theorem}
\begin{proof}
When $k=0,$ $H\Lambda^0(\Omega)=H^1(\Omega)$ and $\mathcal{K}^k_1v=0, \mathcal{K}^k_2v=v.$
When $k=n$, $d^{n-1}$ is identified with the divergence operator $\divg$, see Section \ref{secex}. In this case, let $\mathcal{K}^k_2v=0$ and $\mathcal{K}^k_1$ be the $H^1$-regular right inverse of $\divg,$ i.e., $d\mathcal{K}^k_1v=v$, see Theorem 2.4 in \cite{AFW2006} for details.

Assume $1\leq k\leq n-1.$ Let $\overline{\Omega}$ be a compact subset of a ball $B\subset\mathbb{R}^n$.
There exists a linear bounded extension operator $E: H\Lambda^k(\Omega)\rightarrow H\Lambda^k(B)$, see e.g., \cite{MMS2008} and Lemma 5 in \cite{DH2014}. Consider the exterior derivative on $B$: $$d_B=d: L^2\Lambda^{k-1}(B)\rightarrow{}L^2\Lambda^{k}(B),\quad D(d_B)=H\Lambda^{k-1}(B).$$
For $v\in H\Lambda^k(\Omega),$ we can take $z\in N(d_B)^\perp\cap H\Lambda^{k-1}(B)$ such that $dz=dEv.$ Due to \eqref{crange}, Theorem \ref{adjoint}, and $\delta\circ\delta=0$, we have $$N(d_B)^{\perp}=R(d_B^*)=\delta(\mathring{H}^*\Lambda^k(B))\subset \mathring{H}^*\Lambda^{k-1}(B).$$
Therefore $z\in\mathring{H}^*\Lambda^{k-1}(B)\cap H\Lambda^{k-1}(B)\subset H^1\Lambda^k(B)$, where the following Sobolev embedding  (cf.~\cite{Gaffney1955}) is used:
\begin{align*}
   \|z\|_{H^1(B)}\lesssim\|z\|_{L^2(B)}+\|dz\|_{L^2(B)}+\|\delta z\|_{L^2(B)}.
\end{align*}
Using the Poincar\'e inequality \eqref{Poincare} and $\delta z=0$, the above estimate reduces to
\begin{align}\label{bdz}
   \|z\|_{H^1(B)}\lesssim\|dz\|_{L^2(B)}=\|dEv\|_{L^2(B)}\lesssim\|v\|_{H\Lambda^k(\Omega)}.
\end{align} Since $B$ is contractible and $d(Ev-z)=0,$ there exists $\varphi\in H\Lambda^{k-1}(B),$ such that $d\varphi=Ev-z.$ Furthermore, $\varphi$ can be chosen in $H^1\Lambda^{k-1}(B)$ as in the proof of $H^1$-regularity of $z$. In addition, the $H^1$-norm of $\varphi$ can be controlled by
\begin{align}\label{bdvarphi}
    \|\varphi\|_{H^1(B)}\lesssim\|d\varphi\|_{L^2(B)}=\|Ev-z\|_{L^2(B)}\lesssim\|v\|_{H\Lambda^k(\Omega)}.
\end{align}
Taking $\mathcal{K}^k_1v=\varphi|_{\Omega}, \mathcal{K}^k_2v=z|_{\Omega}$ and using \eqref{bdz} and \eqref{bdvarphi} completes the proof.
\end{proof}
Although $\pi_\ell$ is $V$-bounded, $\pi_\ell v$ is not defined for all $v\in L^2\Lambda^k(\Omega)$. Instead, Christiansen and Winther \cite{CW2008} constructed an $L^2$-bounded smoothed cochain projection $\bar{\pi}_\ell$, i.e., $\bar{\pi}^k_\ell: L^2\Lambda^k(\Omega)\rightarrow V_\ell^k$ is a linear operator for each $k$, $\bar{\pi}_\ell$ commutes with $d$, $\bar{\pi}^k_\ell|_{V_\ell^k}=\text{id}$, and $$\|\bar{\pi}^k_\ell\|:=\sup_{v\in L^2\Lambda^k(\Omega),\|v\|=1}{\|\bar{\pi}^k_\ell v\|}\lesssim1.$$ Built upon $\bar{\pi}_\ell$ and $\mathcal{K}_1, \mathcal{K}_2$, we immediately obtain a discrete version of Theorem \ref{rd} and compact operators which are crucial for proving quasi-orthogonality.
\begin{corollary}\label{disrd}
Let $\Omega$ be a bounded Lipschitz domain in $\mathbb{R}^n.$ For each index $k$, let $$i_1: H^1\Lambda^{k-1}(\Omega)\hookrightarrow L^2\Lambda^{k-1}(\Omega),\quad i_2: H^1\Lambda^{k}(\Omega)\hookrightarrow L^2\Lambda^{k}(\Omega)$$
denote the natural inclusions. Then $\bar{\mathcal{K}}^k_{1}=i_1\circ\mathcal{K}^k_1: H\Lambda^{k}(\Omega)\rightarrow L^2\Lambda^{k-1}(\Omega)$ and $\bar{\mathcal{K}}^k_2=i_2\circ\mathcal{K}^k_2: H\Lambda^{k}(\Omega)\rightarrow L^2\Lambda^{k}(\Omega)$ are compact operators. In addition, for $v\in V^k_\ell,$
\begin{equation*}
    v=d\bar{\pi}_\ell\bar{\mathcal{K}}^k_{1}v+\bar{\pi}_\ell\bar{\mathcal{K}}^k_{2}v.
\end{equation*}
\end{corollary}
\begin{proof}
Using the Rellich--Kondrachov lemma, the inclusions $i_1$ and $i_2$ are compact operators. Since $\mathcal{K}^k_1$ and $\mathcal{K}^k_2$ are bounded, the compositions $\bar{\mathcal{K}}^k_{1}=i_1\circ\mathcal{K}^k_1$, $\bar{\mathcal{K}}^k_2=i_2\circ\mathcal{K}^k_2$ are still compact operators.

For $v\in V_\ell^k$, $v=d\bar{\pi}_\ell\bar{\mathcal{K}}^k_{1}v+\bar{\pi}_\ell\bar{\mathcal{K}}^k_{2}v$ directly follows from Theorem \ref{rd}, $v=\bar{\pi}_\ell v$ and that $\bar{\pi}_\ell$ commutes with the exterior derivative $d$.
\end{proof}
\subsection{Error indicator}
On the de Rham complex, we still use $\|\cdot\|_V$ to denote the  $H\Lambda^l(\Omega)$ norm for some $l$. Let $$H^{1}\Lambda^l(\Tl):=\{\omega\in L^{2}\Lambda^l(\Omega): \omega|_T\in H^{1}\Lambda^l(T)\text{ for all } T\in\Tl\}$$ and $\CS_\ell$ be the set of $(n-1)$-faces in $\Tl$. For each interior face $S\in\Sl$ and $\omega\in H^{1}\Lambda^l(\Tl)$, let $\lr{\tr\omega}|_S:=\tr_S(\omega|_{T_{1}})-\tr_S(\omega|_{T_{2}})$ denote the jump of $\tr\omega$ on $S$, where $T_{1}$ and $T_{2}$ are the two simplexes sharing $S$ as an $(n-1)$-face. On each boundary face $S$, $\lr{\tr\omega}:=\tr_S\omega$.

Let $\bar{\delta}_{ij}=0$ if $i=j$, otherwise $\bar{\delta}_{ij}=1$. For $1\leq k\leq n-1,$ let $f\in H^{1}\Lambda^{k}(\Tl)$ and
\begin{equation*}
\begin{aligned}
&{\CE}_{\ell}(T)=\bar{\delta}_{k1}h_T^2\|\delta\CC\sigma_\ell\|^2_T+h_T\|\lr{\tr\star\CC\sigma_\ell}\|^2_{\partial T}\\
&\quad+h_T^2\|\CC\sigma_\ell-\delta u_\ell\|_T^2+h_T\|\lr{\tr\star u_\ell}\|_{\partial T}^2+h_T\|\lr{\tr\star du_\ell}\|^2_{\partial T}\\
&\quad+h_T^2\|f-d\sigma_\ell-\delta du_\ell\|^2_{T}+h_T^2\|\delta(f-d\sigma_\ell)\|^2_T+h_T\|\lr{\tr\star(f-d\sigma_\ell)}\|^2_{\partial T}.
    \end{aligned}
\end{equation*}
When $k=n$, let
\begin{equation}\label{estimatorkn}
    \begin{aligned}
&{\CE}_{\ell}(T)=h_T^2\|\delta\CC\sigma_\ell\|^2_T+h_T\|\lr{\tr\star\CC\sigma_\ell}\|^2_{\partial T}+\|f-d\sigma_\ell\|^2_T\\
&\quad+h_T^2\|\CC\sigma_\ell-\delta u_\ell\|_T^2+h_T\|\lr{\tr\star u_\ell}\|_{\partial T}^2.
\end{aligned}
\end{equation}
The estimator $\CE_{\ell}=\sum_{T\in\Tl}\CE_{\ell}(T)$ controls the error $\|\sigma-\sigma_\ell\|^2_{V_\CC}+\|u-u_\ell\|^2_V$.
$\CE_{\ell}$ for the standard Hodge Laplace equation was first introduced in \cite{DH2014}. Let  $$\CE_\ell(\mathcal{M})=\sum_{T\in\mathcal{M}}\CE_\ell(T),\quad\mathcal{M}\subseteq\Tl.$$
Define the enriched collection of refinement elements:
\begin{equation*}
    \mathcal{R}_\ell=\{T\in\Tl: T\cap T^\prime\neq\emptyset\text{ for some }T^\prime\in\CT_{\ell}\backslash\CT_{\ell+1}\}.
\end{equation*}
The next theorem confirms the discrete and continuous reliablity of $\CE_\ell.$ It is noted that Demlow \cite{Demlow2017} used similar technique when deriving the discrete reliability of his AFEM for computing harmonic forms.
\begin{theorem}[discrete and continuous upper bounds]\label{disupper}
Assume $f\in H^1\Lambda^k(\Tl)$ with $1\leq k\leq n-1$ or $f\in L^2\Lambda^n(\Omega)$ when $k=n$. There exist a constant $C_{\emph{up}}$ depending solely on $C_\CC, \CT_0$ and $\Omega$, such that
\begin{align}
    \|\sigma_{\ell+1}-\sigma_\ell\|_{V_\CC}^2+\|u_{\ell+1}-u_\ell\|_V^2&\leq C_{\emph{up}}\CE_{\ell}({\mathcal{R}}_\ell),\label{disbd}\\
    \|\sigma-\sigma_{\ell}\|_{V_\CC}^2+\|u-u_{\ell}\|_V^2&\leq C_{\emph{up}}\CE_{\ell}.\label{ctsbd}
\end{align}
\end{theorem}
\begin{proof}
For $\tau\in V^{k-1}_{\ell+1}$ with $\|\tau\|_V=1$, let $\tau=d\varphi_1+z_1$ , where $\varphi_1=\mathcal{K}^{k-1}_{1}\tau$  and $z_1=\mathcal{K}^{k-1}_{2}\tau$ in Theorem \ref{rd} satisfying the following bounds
$$\|\varphi_1\|_{H^1(\Omega)}+\|z_1\|_{H^1(\Omega)}\lesssim1.$$
Similarly, for $v\in V^k_{\ell+1}$ with $\|v\|_V=1$, we have $v=d\varphi_2+z_2$ with $$\|\varphi_2\|_{H^1(\Omega)}+\|z_2\|_{H^1(\Omega)}\lesssim1.$$
Thanks to the local property \eqref{local},
\begin{equation}\label{support}
    \text{supp}(\tau-\pi_\ell\tau)\subseteq\bigcup_{T\in\mathcal{R}_\ell}T,\quad\text{supp}(v-\pi_\ell v)\subseteq\bigcup_{T\in\mathcal{R}_\ell}T.
\end{equation}
Since $\pi_{\ell+1}$ is a cochain projection, we have $\tau=\pi_{\ell+1}\tau=d\pi_{\ell+1}\varphi_1+\pi_{\ell+1}z_1$ and $v=\pi_{\ell+1}v=d\pi_{\ell+1}\varphi_2+\pi_{\ell+1}z_2$. Hence
\begin{equation}\label{tautaul}
    \tau-\pi_{\ell}\tau=dE_{\varphi_1}+E_{z_1},\quad v-\pi_{\ell}v=dE_{\varphi_2}+E_{z_2},
\end{equation}
where $$E_{\varphi_i}=\pi_{\ell+1}\varphi_i-\pi_{\ell}\pi_{\ell+1}\varphi_i,\quad E_{z_i}=\pi_{\ell+1}z_i-\pi_{\ell}\pi_{\ell+1}z_i,\quad i=1, 2.$$
Here $E_{\varphi_1}=0$ when $k=1.$
Using \eqref{support}, \eqref{tautaul}, and the integration by parts formula \eqref{IP}, we have
\begin{equation*}
\begin{aligned}
&\ab{\CC\sigma_\ell}{\tau-\pi_\ell\tau}-\ab{d(\tau-\pi_\ell\tau)}{u_\ell}\\
&\quad=\sum_{T\in{\mathcal{R}}_\ell}\big(\ab{\CC\sigma_\ell}{dE_{\varphi_1}}_T+\ab{\CC\sigma_\ell}{E_{z_1}}_T-\ab{dE_{z_1}}{u_\ell}_T\big)\\
&\quad=\sum_{T\in{\mathcal{R}}_\ell}\big(\ab{\delta\CC\sigma_\ell}{E_{\varphi_1}}_T+\ab{\CC\sigma_\ell-\delta u_\ell}{E_{z_1}}_T\\
&\qquad+\int_{\partial T}\tr E_{\varphi_1}\wedge\tr\star\CC\sigma_\ell-\int_{\partial T}\tr E_{z_1}\wedge\tr\star u_\ell\big).
\end{aligned}
\end{equation*}
Let $\CS(\mathcal{R}_\ell)=\{S\in\Sl: S\subset\partial T\text{ for some }T\in\mathcal{R}_\ell\}$. Since $E_{\varphi_1}\in V_{\ell+1}^{k-2}$ and $E_{z_1}\in V_{\ell+1}^{k-1}$ are finite element differential forms, their traces $\tr E_{\varphi_1}$ and $\tr E_{z_1}$ are well-defined polynomials on each side $S\in\CS_\ell$. Therefore using the previous equation, we obtain
\begin{equation}\label{rhssigmau1}
\begin{aligned}
&\ab{\CC\sigma_\ell}{\tau-\pi_\ell\tau}-\ab{d(\tau-\pi_\ell\tau)}{u_\ell}\\
&\quad=\sum_{T\in{\mathcal{R}}_\ell}\big(\ab{\delta\CC\sigma_\ell}{E_{\varphi_1}}_T+\ab{\CC\sigma_\ell-\delta u_\ell}{E_{z_1}}_T\big)\\
&\qquad+\sum_{S\in\CS(\mathcal{R}_\ell)}\big(\int_S \tr E_{\varphi_1}\wedge\lr{\tr\star\CC\sigma_\ell}-\int_S\tr E_{z_1}\wedge\lr{\tr\star u_\ell}\big).
\end{aligned}
\end{equation}
By the same argument, it holds that for $1\leq k\leq n-1$,
\begin{equation}\label{rhssigmau2}
\begin{aligned}
&\ab{f-d\sigma_\ell}{v-\pi_\ell v}-\ab{du_\ell}{d(v-\pi_\ell v)}\\
&\quad=\sum_{T\in{\mathcal{R}}_\ell}\big(\ab{f-d\sigma_\ell}{dE_{\varphi_2}+E_{z_2}}_T-\ab{du_\ell}{dE_{z_2}}_T\big)\\
&\quad=\sum_{T\in{\mathcal{R}}_\ell}\big(\ab{\delta(f-d\sigma_\ell)}{E_{\varphi_2}}_T+\ab{f-d\sigma_\ell-\delta du_\ell}{E_{z_2}}_T\big)\\
&\qquad+\sum_{S\in\CS(\mathcal{R}_\ell)}\big(\int_S\tr E_{\varphi_2}\wedge\lr{\tr\star(f-d\sigma_\ell)}-\int_S\tr E_{z_2}\wedge\lr{\tr\star du_\ell}\big).
\end{aligned}
\end{equation}
Combining \eqref{rhssigmau1}, \eqref{rhssigmau2} with Lemma \ref{absigmau} and using the Cauchy--Schwarz inequality, we obtain for $1\leq k\leq n-1,$
\begin{equation*}
    \begin{aligned}
        &\|\sigma_{\ell+1}-\sigma_\ell\|_{V_\CC}+\|u_{\ell+1}-u_\ell\|_V\\
        &\quad\lesssim\CE_\ell(\mathcal{R}_\ell)^\frac{1}{2}\big(\sum_{i=1}^2h_T^{-2}\|E_{\varphi_i}\|_T^2+h_T^{-1}\|\tr E_{\varphi_i}\|^2_{\partial T}+h_T^{-2}\|E_{z_i}\|_T^2+h_T^{-1}\|\tr E_{z_i}\|^2_{\partial T}\big)^\frac{1}{2}.
    \end{aligned}
\end{equation*}
The discrete upper bound \eqref{disbd} then follows together with  the approximation property \eqref{Gapproxpi} and the bounds $\sum_{i=1}^2\|\varphi_i\|_{H^1(\Omega)}+\|z_i\|_{H^1(\Omega)}\lesssim1$.

When $k=n$, we simply have $du_\ell=0$, $\ab{f-d\sigma_\ell}{\pi_\ell v}=0$, and thus
\begin{equation}\label{rhssigmau3}
\ab{f-d\sigma_\ell}{v-\pi_\ell v}-\ab{du_\ell}{d(v-\pi_\ell v)}=\ab{f-d\sigma_\ell}{v}.
\end{equation}
Combining it with Lemma \ref{absigmau}, \eqref{rhssigmau1}, \eqref{Gapproxpi} still yields \eqref{disbd}.

Let $\CT_{\ell+1}$ be a uniform refinement of $\CT_\ell$. In this case $\CE(\mathcal{R}_\ell)=\CE_\ell.$ In addition, $(\sigma_{\ell+1},u_{\ell+1})$ converges to $(\sigma,u)$ in $V^{k-1}\times V^k$ as $\max_{T\in\CT_{\ell+1}}h_T\rightarrow0.$
Therefore passing to the limit in \eqref{disbd} yields the continuous upper bound \eqref{ctsbd}.
\end{proof}
For an integer $p\geq0,$ let $Q^p_T$ denote the $L^{2}$-projection onto $\mathcal{P}_p\Lambda^l(T)$ and $Q^p_{\partial T}$ the $L^{2}$-projection onto $\mathcal{P}_p\Lambda^l(\partial T)$ with appropriate $l$. Here $\mathcal{P}_p\Lambda^l(\partial T)$ is the space of $l$-forms on $\partial T$ whose restriction to each face of $T$ are polynomial $l$-forms of degree $\leq p$. Let $\Pi^p_T=\text{id}-Q^p_T$ and $\Pi^p_{\partial T}=\text{id}-Q^p_{\partial T}$. The efficiency of $\CE_{\ell}$ directly follows from the Verf\"urth bubble function technique used in \cite{DH2014} and the proof is skipped.
\begin{theorem}[efficiency]\label{eff}
Assume $\CC (V^{k-1}_\ell)\subseteq\mathcal{P}_{q}\Lambda^{k-1}(\CT_\ell)$ for some $q\geq0.$ There exists a constant $C_{\emph{low}}>0$ depending solely on $p, q, \CC, \CT_0, \Omega$, such that
\begin{equation*}
C_{\emph{low}}\CE_{\ell}\leq\|\sigma-\sigma_\ell\|^{2}_{V_\CC}+\|u-u_\ell\|_V^2+\osc_{\ell}^2(f),
\end{equation*}
where $\osc_{\ell}^2(f)=\sum_{T\in\Tl}\osc_{\ell}^2(f,T)$ with
\begin{align*}
&\osc^{2}_{\ell}(f,T)=\bar{\delta}_{kn}\{h_T^{2}\|\Pi^p_T( f-d\sigma_\ell-\delta du_\ell)\|_T^{2}\\
&\qquad+h_T^{2}\|\Pi^p_T\delta( f-d\sigma_\ell)\|_T^{2}+h_T\|\Pi^p_{\partial T}\lr{\tr\star (f-d\sigma_\ell)}\|_{\partial T}^{2}\}.
\end{align*}
\end{theorem}
From the definitions of $\CE_\ell$ and $\osc_\ell$, it can be observed that
\begin{equation}\label{dom}
    \osc_\ell(f)\leq\CE_\ell.
\end{equation}

\section{quasi-optimality}\label{secqo}
The adaptive algorithm \textsf{AMFEM} is based on the aforementioned
$\textsf{SOLVE}\rightarrow\textsf{ESTIMATE}\rightarrow\textsf{MARK}\rightarrow\textsf{REFINE}$ feedback loop. In the routine \textsf{REFINE}, the collection of marked elements are subdivided by the newest vertex bisection (\text{NVB}) based on tagged simplices and  then a subroutine $\textsf{COMPLETE}$ is used for removing newly created hanging nodes, see, e.g., \cite{Stevenson2008,BDD2004,Traxler1997,Maubach1995}. For simplicity of presentation, the error tolerance is set to be $0$ so that $\textsf{AMFEM}$ produces an infinite sequence $\{(\sigma_\ell,u_\ell,\Tl)\}_{\ell\geq0}$. In order to compute the estimators $\{\CE_\ell\}_{\ell\geq0}$, we assume that $f\in H^1\Lambda^k(\CT_0)$ when $1\leq k\leq n-1,$ that is, the discontinuity of $f$ is aligned with the initial mesh $\CT_0$.

\begin{algorithm}
[\textsf{AMFEM}] Input an initial mesh $\mathcal{T}_{0}$ and a marking parameter $\theta\in(0,1)$. Set $\ell=0$.
\begin{itemize}
\item[]\textsf{SOLVE}: Solve \eqref{DGHL} on $\mathcal{T}_{\ell}$ to obtain the finite element solution $(\sigma_{\ell},u_{\ell})$.
\item[]\textsf{ESTIMATE}: Compute error indicators $\{\CE_{\ell}(T)\}_{T\in\Tl}$ and $\CE_{\ell}=\sum_{T\in\Tl}\CE_\ell(T).$
If $\CE_{\ell}=0$, let $(\sigma_j,u_j)=(\sigma_{{\ell}},u_{{\ell}})$ and $\CT_j=\CT_{\ell}$ for all $j\geq{\ell}$; return.
\item[]\textsf{MARK}: Select a subset $\mathcal{M}_{\ell}$ of $\mathcal{T}_{\ell}$ with  ${\CE}_\ell(\mathcal{M}_{\ell})\geq\theta\CE_\ell.$
\item[]\textsf{REFINE}: Subdivide all elements in $\mathcal{M}_{\ell}$ by NVB and then remove newly created hanging nodes by \textsf{COMPLETE} to obtain the conforming mesh $\mathcal{T}_{\ell+1}$. Set $\ell=\ell+1$. Go to \textsf{SOLVE}.
\end{itemize}
\end{algorithm}

In the following  we use the subscripts $\ell$, $\ell+1$ to indicate quantities on nested meshes. Sometimes $\ell$ and $\ell + 1$ also refer to the levels in \textsf{AMFEM} and in such cases we will explicitly point this out so that no confusion arises from such notation.

\subsection{Contraction}
In this subsection, we prove the contraction of \textsf{AMFEM}. To this end, we first prove a weak convergence result for Petrov--Galerkin methods. Given Hilbert spaces $\mathcal{U}, \mathcal{V}$ equipped with norm $\vertiii{\cdot}$, a continuous bilinear form $\mathcal{B}: \mathcal{U}\times \mathcal{V}\rightarrow\mathbb{R}$ and a continuous linear functional $F: \mathcal{V}\rightarrow\mathbb{R},$
consider the variational formulation:  Find $U\in\mathcal{U}$ such that
\begin{equation*}
    \mathcal{B}(U,V)=F(V),\quad V\in\mathcal{V}.
\end{equation*}
Given subspaces $\mathcal{U}_\ell\subseteq\mathcal{U}$ and $\mathcal{V}_\ell\subseteq\mathcal{V}$, the abstract Petrov--Galerkin method seeks $U_\ell\in\mathcal{U}_\ell$ such that
\begin{equation}\label{PG}
    \mathcal{B}(U_\ell,V)=F(V),\quad V\in\mathcal{V}_\ell.
\end{equation}
The method \eqref{PG} is well-posed provided
\begin{align*}
    &\inf_{W\in\mathcal{U}_\ell,\vertiii{W}=1}\sup_{V\in\mathcal{V}_\ell,\vertiii{V}=1}\mathcal{B}(W,V)\\
    &\quad=\inf_{V\in\mathcal{V}_\ell,\vertiii{V}=1}\sup_{W\in\mathcal{U}_\ell,\vertiii{W}=1}\mathcal{B}(W,V)\ge\beta>0,
\end{align*}
where $\beta$ is a constant independent of $\ell.$
We say $x_\ell\rightharpoonup x$ in $\mathcal{U}$ if $\{x_\ell\}_{\ell\geq0}$ weakly converges to $x$ in $\mathcal{U}$ as $\ell\rightarrow\infty$. The next theorem shows that the normalized error of \eqref{PG} always weakly converges to $0.$
\begin{theorem}[weak convergence]\label{weakPG}
Let $\{\mathcal{U}_\ell\}_{\ell\geq0}$ and $\{\mathcal{V}_\ell\}_{\ell\geq0}$ be sequences of subspaces of $\mathcal{U}$ and $\mathcal{V}$, respectively. Assume $\mathcal{U}_\ell\subseteq\mathcal{U}_{\ell+1}$ and $\mathcal{V}_\ell\subseteq\mathcal{V}_{\ell+1}$ for all $\ell\geq0$. Let {$U_\ell$ be the solution of \eqref{PG} and}
\begin{equation*}
    \xi_\ell=
    \left\{\begin{aligned}
        &\frac{U_{\ell+1}-U_\ell}{\vertiii{U_{\ell+1}-U_\ell}}&& \text{if } \vertiii{U_{\ell+1}-U_\ell}\neq0,\\
        &0&&\text{if }\vertiii{U_{\ell+1}-U_\ell}=0.
    \end{aligned}\right.
\end{equation*}
We have
\begin{align*}
    \xi_\ell\rightharpoonup0\text{ in }\mathcal{U}.
\end{align*}
\end{theorem}
\begin{proof}
Since $\vertiii{\xi_\ell}\leq1$, there exists a weakly convergent subsequence $\{\xi_{\ell_j}\}_{j\geq0}$:
\begin{equation*}
    \xi_{\ell_j}\rightharpoonup\bar{\xi}\text{ in }\mathcal{U}.
\end{equation*}
For any $\ell\geq0$ and $V_\ell\in\mathcal{V}_\ell$, it follows from the continuity of $\mathcal{B}$ that $\mathcal{B}(\cdot,V_\ell)$ is a continuous linear functional on $\mathcal{U}.$ Hence using the weak convergence and the Galerkin orthogonality
\begin{align*}
    \mathcal{B}(U_{\ell_{j+1}}-U_{\ell_j},V_\ell)=0\text{ when }\ell_j\geq\ell,
\end{align*}
we have
\begin{equation}\label{Bxibar1}
    \mathcal{B}(\bar{\xi},V_\ell)=\lim_{j\rightarrow\infty}\mathcal{B}({\xi}_{\ell_j},V_\ell)=0\text{ for all }V_\ell\in\mathcal{V}_\ell\text{ and }\ell\geq0.
\end{equation}
Given $\mathcal{W}\subseteq\mathcal{U}$ (resp.~$\mathcal{V}$), let $\overline{\mathcal{W}}\subseteq\mathcal{U}$ (resp.~$\mathcal{V}$) denote the closed subspace spanned by $\mathcal{W}.$ We consider
$$\mathcal{U}_\infty=\overline{\bigcup_{\ell\geq0}\mathcal{U}_\ell},\quad \mathcal{V}_\infty=\overline{\bigcup_{\ell\geq0}\mathcal{V}_\ell}.$$
Since $\mathcal{U}_\infty$ is closed, convex and $\{\xi_{\ell_j}\}_{j\geq0}\subset\mathcal{U}_\infty$,  the weak limit $\bar{\xi}$ is contained in $\mathcal{U}_\infty.$ The next discrete inf-sup condition is proved by Morin, Siebert, and Veeser, see Equation (4.5) of  \cite{MSV2008}; see also \cite{BV1984} for the symmetric case.
\begin{align}\label{infsupinfty}
    \inf_{W\in\mathcal{U}_\infty,\vertiii{W}=1}\sup_{V\in\mathcal{V}_\infty,\vertiii{V}=1}\mathcal{B}(W,V)\geq\beta>0.
\end{align}
On the other hand, it follows from \eqref{Bxibar1} and the continuity of $\mathcal{B}$ that
\begin{equation}\label{Bxibar}
    \mathcal{B}(\bar{\xi},V)=0\text{ for all }V\in\mathcal{V}_\infty.
\end{equation}
Therefore using  \eqref{infsupinfty} and \eqref{Bxibar}, we obtain
\begin{align*}
    \vertiii{\bar{\xi}}\leq\beta^{-1}\sup_{V\in\mathcal{V}_\infty,\vertiii{V}=1}\mathcal{B}(\bar{\xi},V)=0.
\end{align*}
We have indeed shown that each subsequence of $\{\xi_\ell\}$ has a further subsequence that weakly converges to $0$, which apparently implies $\xi_\ell\rightharpoonup0$ in $\mathcal{U}$.
\end{proof}
\begin{remark}
Let the sequence $\{\mathcal{U}_\ell\}_{\ell\geq0}$ be produced by some AFEM and
\begin{equation*}
    \zeta_\ell=
    \left\{\begin{aligned}
        &\frac{U-U_\ell}{\vertiii{U-U_\ell}}&&\text{if } \vertiii{U-U_\ell}\neq0,\\
        &0&&\text{if }\vertiii{U-U_\ell}=0.
    \end{aligned}\right.
\end{equation*}
Under very mild assumptions, it has been shown in \cite{MSV2008} that $\mathcal{U}_\infty=\mathcal{U}$, see also \cite{FFP2014} for AFEMs using D\"orfler marking. In this case, $\zeta_\ell\in\mathcal{U}_\infty$, and the weak convergence
$\zeta_\ell\rightharpoonup0\text{ in }\mathcal{U}$
follows from the same proof of Theorem \ref{weakPG}.
\end{remark}
For each $\ell\geq0$, let \begin{align*}
    &e_{d\sigma,\ell}=\|d(\sigma-\sigma_{\ell})\|^2,\quad e_{\sigma,\ell}=\|\sigma-\sigma_{\ell}\|_{V_\CC}^2,\quad e_{u,\ell}=\|u-u_{\ell}\|_V^2,\\ &E_{d\sigma,\ell}=\|d(\sigma_{\ell+1}-\sigma_{\ell})\|^2,\quad E_{\sigma,\ell}=\|\sigma_{\ell+1}-\sigma_{\ell}\|_{V_\CC}^2,\quad E_{u,\ell}=\|u_{\ell+1}-u_{\ell}\|_V^2,
\end{align*}
$e_\ell=e_{\sigma,\ell}+e_{u,\ell}$ and $ E_\ell=E_{\sigma,\ell}+E_{u,\ell}.$
Let $$E_\ell^{-\frac{1}{2}}(\sigma_{\ell+1}-\sigma_\ell):=0\text{ and }E_\ell^{-\frac{1}{2}}(u_{\ell+1}-u_\ell):=0\text{ if }E_\ell=0.$$ Theorem \ref{weakPG} immediately implies the following weak convergence result of \eqref{DGHL}.
\begin{corollary}\label{weakHL}
For $1\leq k\leq n,$
it holds that
\begin{align*}
    &E_\ell^{-\frac{1}{2}}(\sigma_{\ell+1}-\sigma_\ell)\rightharpoonup0\text{ in }H\Lambda^{k-1}(\Omega),\\
    &E_\ell^{-\frac{1}{2}}(u_{\ell+1}-u_\ell)\rightharpoonup0\text{ in }H\Lambda^{k}(\Omega).
\end{align*}
\end{corollary}
The next lemma deals with error reduction on two nested meshes $\Tl\leq\CT_{\ell+1}.$
\begin{theorem}\label{relations}
For $\ell\geq0$ and $\varepsilon\in(0,1)$, it holds that
\begin{align}
     e_{d\sigma,\ell+1}&= e_{d\sigma,\ell}-E_{d\sigma,\ell},\label{qo1}\\
    (1-\varepsilon){e}_{\sigma,\ell+1}&\leq {e}_{\sigma,\ell}-(1-\varepsilon){E}_{\sigma,\ell}+C_{\varepsilon}E_{d\sigma,\ell}\label{qo2},
\end{align}
where $C_\varepsilon=\varepsilon^{-1}{C}_\sigma,$ and ${C}_\sigma$ depends only on $C_\CC, \CT_0$ and $\Omega.$
In addition, for any $\varepsilon>0$, there exists $\ell_0=\ell_0(\varepsilon)\in\mathbb{N}$, such that whenever $\ell\geq\ell_0,$
\begin{align}
    e_{u,\ell+1}\leq e_{u,\ell}-\frac{1}{2}E_{u,\ell}+\frac{1}{2}E_{\sigma,\ell}+\varepsilon e_{\ell+1}.\label{qo3}
\end{align}
\end{theorem}
\begin{proof}
The relation \eqref{qo1} follows from $\ab{d(\sigma-\sigma_{\ell+1})}{d(\sigma_\ell-\sigma_{\ell+1})}=0$; and the inequality \eqref{qo2} has been proved in Lemma 4.1 of \cite{YL2019}.

Using  Corollary \ref{disrd}, we have
\begin{align}\label{decomp}
    v_\ell:=E_\ell^{-\frac{1}{2}} (u_{\ell}-u_{\ell+1})=d\varphi_\ell+z_\ell,
\end{align} where $$\varphi_\ell=\bar{\pi}_{\ell+1}\bar{\mathcal{K}}^k_{1}v_\ell\in V_{\ell+1}^{k-1},\quad z_\ell=\bar{\pi}_{\ell+1}\bar{\mathcal{K}}^k_{2}v_\ell\in V_{\ell+1}^{k}.$$ On the other hand, it follows from Corollary \ref{weakHL} that $$v_\ell\rightharpoonup0\text{ in }H\Lambda^k(\Omega).$$ Therefore the compact operators $$\bar{\mathcal{K}}^k_{1}: H\Lambda^{k}(\Omega)\rightarrow L^2\Lambda^{k-1}(\Omega),\quad \bar{\mathcal{K}}^k_{2}: H\Lambda^{k}(\Omega)\rightarrow L^2\Lambda^{k}(\Omega)$$ produce strongly  convergent sequences in the $L^2$-norm:
\begin{align*}
    \|\bar{\mathcal{K}}^k_1v_\ell\|\rightarrow0\text{ and }\|\bar{\mathcal{K}}^k_2v_\ell\|\rightarrow0\text{ as }\ell\rightarrow\infty.
\end{align*}
Due to $\|\bar{\pi}_{\ell+1}\|\lesssim1$, we obtain
\begin{align*}
    \|\varphi_\ell\|\rightarrow0\text{ and }\|z_\ell\|\rightarrow0\text{ as }\ell\rightarrow\infty.
\end{align*}
In particular, for any $\varepsilon>0$, there exists $\ell_0=\ell_0(\varepsilon),$ such that
\begin{align}\label{small}
    \|\CC^\frac{1}{2}\varphi_\ell\|\leq\frac{\varepsilon^\frac{1}{2}}{4}\text{ and }\|z_\ell\|\leq\frac{\varepsilon^\frac{1}{2}}{4}\text{ whenever } \ell\geq\ell_0.
\end{align} Using \eqref{decomp}, \eqref{small}, and
\begin{align*}
    \ab{\CC(\sigma-\sigma_{\ell+1})}{\varphi_\ell}=\ab{d\varphi_\ell}{u-u_{\ell+1}},
\end{align*}we have for $\ell\geq\ell_0$,
\begin{equation}\label{uortho}
    \begin{aligned}
    &\ab{u-u_{\ell+1}}{u_{\ell}-u_{\ell+1}}=\ab{\CC(\sigma-\sigma_{\ell+1})}{\varphi_\ell}E_\ell^\frac{1}{2}+\ab{u-u_{\ell+1}}{z_\ell}E_\ell^\frac{1}{2}\\
    &\quad\leq\frac{\varepsilon^\frac{1}{2}}{4}\big(\|\CC^\frac{1}{2}(\sigma-\sigma_{\ell+1})\|+\|u-u_{\ell+1}\|\big)E_\ell^\frac{1}{2}\leq\frac{1}{8}E_\ell+\frac{\varepsilon}{4}e_{\ell+1}.
\end{aligned}
\end{equation}
Similarly, for $\ell\geq\ell_0,$
\begin{equation}\label{duortho}
    \begin{aligned}
    &\ab{d(u-u_{\ell+1})}{d(u_{\ell}-u_{\ell+1})}=-\ab{d(\sigma-\sigma_{\ell+1})}{z_\ell}E_\ell^\frac{1}{2}\\
    &\qquad\leq\frac{\varepsilon^\frac{1}{2}}{4}\|d(\sigma-\sigma_{\ell+1})\| E_\ell^\frac{1}{2}\leq\frac{1}{8}E_\ell+\frac{\varepsilon}{8}e_{d\sigma,\ell+1}.
\end{aligned}
\end{equation}
Combining \eqref{uortho}, \eqref{duortho} with
$$e_{u,\ell+1}=e_{u,\ell}-E_{u,\ell}+2\ab{u-u_{\ell+1}}{u_{\ell}-u_{\ell+1}}_V,$$
we obtain \eqref{qo3}. The proof is complete.
\end{proof}
The following estimator reduction is standard and essentially relying on the D\"orfler marking condition used in the routine \textsf{MARK}, see Corollary 3.4 of \cite{CKNS2008}.
\begin{lemma}[estimator reduction]\label{continuity}
Assume that $f\in H^1\Lambda^k(\CT_0)$ when $1\leq k\leq n-1,$ or $f\in \mathcal{P}_r\Lambda^n(\mathcal{T}_0)$ when $k=n.$ Let $\{\mathcal{E}_\ell\}_{\ell\ge0}$ be a sequence of estimators generated by \textsf{AMFEM}. Then
there exist constants $0<\zeta<1$ and $C_{\emph{re}}>0$ depending only on $\theta, n, \CT_0, C_\CC,  \Omega$, such that
\begin{equation*}
{\CE}_{\ell+1}\leq\zeta{\CE}_\ell+C_{\emph{re}}E_{\ell}.
\end{equation*}
\end{lemma}
\begin{remark}\label{remarkestreduction}
Recall in \eqref{pairs} that $V_0^n=\mathcal{P}^-_{r+1}\Lambda^n(\mathcal{T}_0)=\mathcal{P}_r\Lambda^n(\mathcal{T}_0)$, the space of $n$-forms with discontinuous piecewise polynomial coefficients w.r.t.~$\mathcal{T}_0$ of degree at most $r$. If $k=n$ and $f\in \mathcal{P}_r\Lambda^n(\mathcal{T}_0)$, we have $f\in\mathcal{P}_r\Lambda^n(\mathcal{T}_\ell)$ for all $\ell\geq0$, and thus $d\sigma_\ell=f$ by \eqref{DGHL2}. Hence the term $\|f-d\sigma_\ell\|_T$ in $\mathcal{E}_\ell$ vanishes,  see \eqref{estimatorkn}. Otherwise $\|f-d\sigma_\ell\|_T$ in $\mathcal{E}_\ell$ is an obstacle against the construction of the standard reduction property described  in Lemma \ref{continuity}.
\end{remark}

To prove the convergence of AFEMs for nonsymmetric $2^{nd}$ order elliptic problems, the authors in \cite{FFP2014} developed the estimator contraction technique based on the following-type quasi-orthogonality
\begin{align*}
    \vertiii{U-U_{\ell+1}}^2\leq\frac{1}{1-\varepsilon}\vertiii{U-U_\ell}^2-\vertiii{U_{\ell+1}-U_\ell}^2,\quad\ell\geq\underline{\ell}(\varepsilon).
\end{align*}
Adding \eqref{qo2} and \eqref{qo3} in Theorem \ref{relations}, we obtain a weaker quasi-orthogonality
\begin{align}\label{trueqo}
    (1-2\varepsilon)e_{j+1}\leq e_{j}-\frac{1}{4}E_{j}+C_\varepsilon E_{d\sigma,j}\quad\text{ for }j\geq\ell_0=\ell_0(\varepsilon),
\end{align}
whenever $0<\varepsilon\le\frac{1}{4}$. The remainder term $E_{d\sigma,j}$ is estimated using the exact Pythagorean identity \eqref{qo1} in the following analysis.

With the above preparations, we are able to prove the contraction of the estimator $\CE_\ell$ in \textsf{AMFEM} using the quasi-orthogonality result \eqref{trueqo}, the continuous upper bound in Theorem \ref{disupper}, and Lemma \ref{continuity}.
\begin{theorem}[estimator contraction]\label{contraction}
Let the assumption in Lemma \ref{continuity} hold and $\{\mathcal{E}_\ell\}_{\ell\ge0}$ be a sequence of estimators generated by \textsf{AMFEM}. There exist constants $0<\gamma<1$ and $C_{\emph{conv}}>0$ depending solely on $\zeta, C_{\emph{up}}, C_{\emph{re}}, C_\varepsilon$, and $\ell_0=\ell_0(\varepsilon)$ with $$\varepsilon=\min\left(\frac{1-\zeta}{16C_{\emph{up}}C_{\emph{re}}},\frac{1}{4}\right),$$ such that for all $\ell\geq0$ and $m\geq1$, it holds that
\begin{equation*}
\CE_{\ell+m}\leq C_{\emph{conv}}\gamma^m\CE_\ell.
\end{equation*}
\end{theorem}
\begin{proof}
For any $J\geq\ell+1$ and $\alpha=\frac{1-\zeta}{2}$, it follows from Lemma \ref{continuity} and \eqref{ctsbd} that
\begin{align*}
    &\sum_{j=\ell+1}^J\CE_j\leq\sum_{j=\ell+1}^J\big(\zeta\CE_{j-1}+C_{\text{re}}E_{j-1}\big)\\
    &\quad\leq\sum_{j=\ell+1}^J\big\{(\zeta+\alpha)\CE_{j-1}+C_{\text{re}}(E_{j-1}-\alpha C_{\text{re}}^{-1}C_{\text{up}}^{-1}e_{j-1})\big\},
\end{align*}
and thus
\begin{equation}\label{total}
    (1-\zeta-\alpha)\sum_{j=\ell+1}^J\CE_j\leq(\zeta+\alpha)\CE_\ell+C_{\text{re}}\sum_{j=\ell}^{J-1}(E_{j}-\beta e_{j}),
\end{equation}
where $\beta=\alpha C_{\text{up}}^{-1}C^{-1}_{\text{re}}$. If the Pythagorean identity $e_{j+1}=e_j-E_j$ as a consequence of the exact orthogonality  holds, then we could exploit the simple estimate $\sum_{j=\ell}^{J-1}(E_{j}-\beta e_{j})\le\sum_{j=\ell}^{J-1}(e_{j}-e_{j+1})\le e_\ell.$ Unfortunately, such orthogonality does not hold in the case of  indefinite systems and different arguments are needed in the analysis to circumvent this.

In general, to control the sum on the right hand side of \eqref{total}, we use the quasi-orthogonality \eqref{trueqo} with $\varepsilon=\min(\frac{\beta}{8},\frac{1}{4})$ and \eqref{qo1} to obtain that for all $\ell\geq\ell_0,$
\begin{equation}\label{part2}
    \begin{aligned}
        &\sum_{j=\ell}^{J-1}\big(E_{j}-\beta e_{j}\big)\leq\sum_{j=\ell}^{J-1}\big(4e_{j}-4(1-2\varepsilon)e_{j+1}+4C_\varepsilon E_{d\sigma,j}-\beta e_{j}\big)\\
        &\qquad\leq(4-\beta)\sum_{j=\ell}^{J-1}\big(e_{j}-e_{j+1}\big)+4C_\varepsilon\sum_{j=\ell}^{J-1}\big(e_{d\sigma,j}-e_{d\sigma,j+1}\big)\\
        &\qquad\leq(4-\beta)e_\ell+4C_\varepsilon e_{d\sigma,\ell}\leq(4-\beta+4C_\varepsilon)e_\ell.
    \end{aligned}
\end{equation}
A combination of \eqref{total}, \eqref{part2}, and $e_\ell\leq C_{\text{up}}\CE_\ell$ shows that for $\ell\geq\ell_0,$
\begin{equation}\label{sum1}
    (1-\zeta-\alpha)\sum_{j=\ell+1}^J\CE_j\leq{C}_1\CE_\ell,
\end{equation}
where ${C}_1=\zeta+\alpha+(4-\beta+4C_\varepsilon)C_{\text{up}}C_{\text{re}}.$ For $\ell<\ell_0,$ $\CE_\ell=0$ implies that the algorithm \textsf{AMFEM} terminates at step $\ell$ and $\CE_{\ell+1}=\CE_{\ell+2}=\cdots=0.$
Hence we can simply take
$$C_{\text{sup}}:=\max_{0\leq\ell<\ell_0,\CE_\ell\neq0}\CE_{\ell}^{-1}\sum_{j=\ell+1}^{\ell_0}\CE_j<\infty$$
and obtain
\begin{equation}\label{sum2}
    \sum_{j=\ell+1}^{\ell_0}\CE_j\leq C_{\text{sup}}\CE_\ell.
\end{equation}
Let $C_2:=C_{\text{sup}}+(1-\zeta-\alpha)^{-1}{C}_1$. Using \eqref{sum1} and \eqref{sum2}, we have
\begin{equation}\label{sum3}
    \sum_{j=\ell+1}^\infty\CE_j\leq {C}_2\CE_\ell\text{ for all }\ell\geq0.
\end{equation}
Therefore,
\begin{equation}\label{sum4}
    (1+{C}_2^{-1})\sum_{j=\ell+1}^\infty\CE_j\leq\sum_{j=\ell+1}^\infty\CE_j+\CE_\ell=\sum_{j=\ell}^\infty\CE_j\text{ for all  }\ell\geq0.
\end{equation}
Using \eqref{sum4} and \eqref{sum3}, it holds for $m\geq1$ and $\ell\geq0$ that
\begin{align*}
    \CE_{\ell+m}\leq\sum_{j=\ell+m}^\infty\CE_j\leq(1+{C}_2^{-1})^{-m+1}\sum_{j=\ell+1}^\infty\CE_j\leq(1+{C}_2^{-1})^{-m+1}{C}_2\CE_\ell.
\end{align*}
Taking $\gamma=(1+{C}_2^{-1})^{-1}$ and $C_{\text{conv}}=(1+{C}_2^{-1}){C}_2$ completes the proof.
\end{proof}

It follows from \eqref{dom}, \eqref{ctsbd}, and Theorem \ref{eff} that
\begin{align}\label{equiv}
    \CE_\ell\approx\|\sigma-\sigma_{\ell}\|^2_{V_\CC}+\|u-u_{\ell}\|^2_V+\osc_\ell^2(f).
\end{align}
Therefore due to Theorem \ref{contraction}, there exists a constant $\widetilde{C}_{\text{conv}}$ depending only on $C_\CC, \CT_0, \Omega,$ and $C_{\text{conv}}$, such that for $\ell\geq0, m\geq1,$
\begin{align*}
    &\|\sigma-\sigma_{\ell+m}\|^2_{V_\CC}+\|u-u_{\ell+m}\|^2_V+\osc_{\ell+m}^2(f)\\
    &\quad\leq\widetilde{C}_{\text{conv}}\gamma^m\big(\|\sigma-\sigma_{\ell}\|^2_{V_\CC}+\|u-u_{\ell}\|^2_V+\osc_{\ell}^2(f)\big).
\end{align*}

\subsection{Quasi-optimality}
Let $\mathbb{T}$ denote the set of all conforming subtriangulations of $\CT_0$ created by NVB. For $\CT_{\ell}\in\mathbb{T},$ and  $\tau\in V_{\ell}^{k-1}, v\in V_{\ell}^k,$ the total error is
\begin{align*}
    E_{\ell}(\tau,v):=\|\sigma-\tau\|^2_{V_\CC}+\|u-v\|^2_V+\osc^2_\ell(f),
\end{align*}
For $s>0$, we make use of the approximation class
\begin{align*}
    \mathbb{A}_{s}&:=\{(\tau,v,g)\in V^{k-1}\times V^k\times W^k: |(\tau,v,g)|_{\mathbb{A}_s}<\infty\},
\end{align*}
where the semi-norm $|\cdot|_{\mathbb{A}_s}$ is defined as
\begin{align*}
    |(\sigma,u,f)|_{\mathbb{A}_s}:=\sup_{N>0}\{N^s\min_{\CT_\ell\in\mathbb{T},\#\Tl-\#\CT_0\leq N}\min_{(\tau,v)\in V_\ell^{k-1}\times V_\ell^k}E_{\ell}(\tau,v)^\frac{1}{2}\}.
\end{align*}
To specify the dependence of $\CE_\ell$ on  $\sigma_\ell, u_\ell$, let $\CE_\ell(T)=\CE_\ell(\sigma_\ell,u_\ell,T)$ and $\CE_\ell(\tau,v,T)$ be given by replacing $(\sigma_\ell,u_\ell)$ with $(\tau,v)$ in $\CE_\ell(T).$ The next estimator perturbation result is standard, see, e.g., Proposition 3.3 in \cite{CKNS2008}.
There exists a constant $C_\text{stab}$ depending only on $\CT_0$, such that for $\tau\in V_{\ell+1}^{k-1}, v\in V_{\ell+1}^{k}$ and $\varepsilon>0,$
\begin{equation*}
    \begin{aligned}
    &\sum_{T\in\Tl\cap\CT_{\ell+1}}\CE_\ell(\sigma_\ell,u_\ell,T)\leq(1+\varepsilon)\sum_{T\in\Tl\cap\CT_{\ell+1}}\CE_\ell(\tau,v,T)\\
    &\qquad+(1+\varepsilon^{-1})C_\text{stab} \big(\|\sigma_\ell-\tau\|_{V_\CC}^2+\|u_\ell-v\|_V^2\big).
\end{aligned}
\end{equation*}
The marking parameter $\theta$ is required to be below the threshold $$\theta_{*}:=\frac{1}{1+C_{\text{up}}C_\text{stab}},$$
see, e.g., Lemma 5.5 in \cite{FFP2014}.
To prove optimality of AFEMs, Stevenson \cite{Stevenson2007} assumed the collection of marked elements has minimal cardinality:
\begin{equation}\label{markcond}
    \text{ Routine \textsf{MARK} selects a subset $\mathcal{M}_{\ell}$}\text{ with minimal cardinality.}
\end{equation}
Recall that $\{\Tl\}_{\ell\geq0}$ produced by \textsf{AMFEM} is a sequence of meshes generated by \text{NVB}. Assuming a matching condition on the initial mesh $\CT_0$, it has been shown in \cite{Stevenson2008,BDD2004} that the accumulated cardinality of marked elements satisfies
\begin{equation}\label{card}
    \#\mathcal{T}_\ell-\#\mathcal{T}_0\lesssim\sum_{j=0}^{\ell-1}\mathcal{M}_j.
\end{equation}

Now using the contraction in Theorem \ref{contraction}, the lower bound in Theorem \ref{eff}, and the assumptions on \textsf{MARK} and the marking parameter, we obtain the quasi-optimal convergence rate of $\{\CE_\ell\}_{\ell\geq0}$, see Theorem 5.3 in \cite{FFP2014}.
\begin{theorem}[quasi-optimality]\label{qopt}
Let $\{(\sigma_\ell,u_\ell,\Tl)\}_{\ell\geq0}$ be a sequence of finite element solutions and meshes generated by \textsf{AMFEM}. Assume $(\sigma,u,f)\in\mathbb{A}_{s}$,  $0<\theta<\theta_{*}$, \eqref{markcond}, \eqref{card}, and the conditions in Theorem \ref{eff} and Lemma \ref{continuity} hold. Then there exists a constant $C_{\emph{opt}}$ depending only on $\theta, \theta_*, \CT_0, C_\CC, \Omega,$ and $s, \gamma, C_{\emph{conv}}$ such that
\begin{equation*}
\sqrt{\CE}_\ell\leq C_{\emph{opt}}|(\sigma,u,f)|_{\mathbb{A}_s}\big(
\#\CT_\ell-\#\CT_0\big)^{-s}.
\end{equation*}
\end{theorem}
Using the equivalence \eqref{equiv} and Theorem \ref{qopt}, we obtain the quasi-optimal convergence rate of \textsf{AMFEM}, that is, there exists $\widetilde{C}_{\text{opt}}$ independent of $\ell$ such that
\begin{align*}
    \big(\|\sigma-\sigma_{\ell}\|^2_{V_\CC}+\|u-u_{\ell}\|^2_V+\osc_{\ell}^2(f)\big)^\frac{1}{2}\leq\widetilde{C}_{\text{opt}}|(\sigma,u,f)|_{\mathbb{A}_s}(\#\CT_\ell-\#\CT_0)^{-s}.
\end{align*}

\subsection{Modified adaptive algorithms}\label{mamfem}
In this subsection, we focus on \eqref{GHL} with $k=n$, i.e.,
\begin{equation}\label{nGHL}
\begin{aligned}
\ab{\CC\sigma}{\tau}-\ab{d\tau}{u}&=0,\quad&&\tau\in V^{n-1},\\
\ab{d\sigma}{v}&=\ab{f}{v},&& v\in V^n,
\end{aligned}
\end{equation}
where $f\in L^2\Lambda^n(\Omega)$. The corresponding mixed method \eqref{DGHL} reads
\begin{equation}\label{nDGHL}
\begin{aligned}
\ab{\CC\sigma_\ell}{\tau}-\ab{d\tau}{u_\ell}&=0,\quad&&\tau\in V_\ell^{n-1},\\
\ab{d\sigma_\ell}{v}&=\ab{f}{v},&& v\in V_\ell^n.
\end{aligned}
\end{equation}
{Let $Q_\CT$ denote the $L^2$-projection onto $\mathcal{P}_r\Lambda^n(\CT).$ It follows from the second equation in \eqref{nDGHL} and $V_\ell^n=\mathcal{P}_r\Lambda^n(\CT_\ell)$ that $d\sigma_\ell=Q_{\CT_\ell}f.$

One of the conditions in Lemma \ref{continuity} requires $f\in\mathcal{P}_r\Lambda^n(\mathcal{T}_0)$ and the estimator contraction in Theorem \ref{contraction} is not guaranteed if $f\in L^2\Lambda^n(\Omega)\backslash \mathcal{P}_r\Lambda^n(\mathcal{T}_0)$. To remedy the situation, one can assume that $f\in H^1\Lambda^n(\mathcal{T}_0)$ and replace $\|f-d\sigma_\ell\|_T$ by $h\|\nabla f\|_T$ in \eqref{estimatorkn}. Then the corresponding error estimator is $\widetilde{\mathcal{E}}_\ell:=\sum_{T\in\Tl}\widetilde{\mathcal{E}}_\ell(T)$, where
\begin{equation*}
    \begin{aligned}
&\widetilde{\CE}_{\ell}(T)=h_T^2\|\delta\CC\sigma_\ell\|^2_T+h_T\|\lr{\tr\star\CC\sigma_\ell}\|^2_{\partial T}+h_T^2\|\nabla f\|^2_T\\
&\quad+h_T^2\|\CC\sigma_\ell-\delta u_\ell\|_T^2+h_T\|\lr{\tr\star u_\ell}\|_{\partial T}^2.
\end{aligned}
\end{equation*}
Replacing $\mathcal{E}_\ell$ with $\widetilde{\mathcal{E}}_\ell$ in \textsf{AMFEM} yields a modified adaptive algorithm (denoted by \textsf{AMFEM--M}).
Due to the factor $h_T$ in each term of $\widetilde{\mathcal{E}}_\ell(T)$, it is readily checked that
\begin{equation}\label{est2}
\widetilde{\CE}_{\ell+1}\leq\zeta\widetilde{\CE}_{\ell}+C_{\text{re}}E_{\ell},
\end{equation}
where $0<\zeta<1$, $C_{\text{re}}>0$ are described in Lemma \ref{continuity}.
In addition, the reliability
\begin{align}\label{ctsbd2}
    \|\sigma-\sigma_{\ell}\|_{V_\CC}^2+\|u-u_{\ell}\|_V^2\lesssim\mathcal{E}_\ell\lesssim\widetilde{\CE}_{\ell}
\end{align}
follows from \eqref{ctsbd} and the elementary inequality $$\|f-d\sigma_\ell\|_T=\|f-Q_{\mathcal{T}_\ell}f\|_T\lesssim h\|\nabla f\|_T.$$
Using \eqref{est2}, \eqref{ctsbd2} and the same argument in the proof of  Theorem \ref{contraction}, we obtain the contraction of \textsf{AMFEM--M}, i.e., \begin{equation*}
\widetilde{\CE}_{\ell+m}\leq C_{\text{conv}}\gamma^m\widetilde{\CE}_\ell,\quad\text{for all }\ell\geq0,~m\geq1,
\end{equation*}
where $\gamma\in(0,1)$, ${C}_{\text{conv}}>0$  are given in Theorem \ref{contraction}.
However, $\widetilde{\mathcal{E}}_\ell$ is not shown to be a lower bound of $\|\sigma-\sigma_\ell\|^2_{V_\CC}+\|u-u_\ell\|^2$ and we will not pursue the quasi-optimality of \textsf{AMFEM--M}.
}

Now we present another modified adaptive algorithm for \eqref{nGHL} with $f\in L^2\Lambda^n(\Omega)$. The key ingredient is the nonlinear approximation algorithm (denoted by \textsf{APPROX}) in \cite{BD2004,BDD2004} for  pre-processing $f$, see Corollary 5.4 in \cite{BD2004}. Assume that
$$|f|_{\mathcal{A}_s}:=\sup_{N>0}\left(N^s\min_{\mathcal{T}\in\mathbb{P}, \#\mathcal{T}-\#\mathcal{T}_0\leq N}\|f-Q_\mathcal{T}f\|\right)<\infty,$$
where $\mathbb{P}$ is the set of subtriangulations (possibly with hanging nodes) of $\CT_0$ generated by NVB.
Given $\texttt{tol}>0$, the routine $\textsf{APPROX}$  outputs $\widetilde{\mathcal{T}}_0=\textsf{APPROX}(\mathcal{T}_0,f,\texttt{tol})\in\mathbb{P}$ by NVB, which is a refinement of $\mathcal{T}_0$ and realizes the class optimality condition, namely,
\begin{subequations}\label{approxf}
\begin{align}
    &\|f-Q_{\widetilde{\mathcal{T}}_0}f\|\leq\texttt{tol},\label{approxfa}\\
    &\#\widetilde{\mathcal{T}}_0-\#\mathcal{T}_0\lesssim|f|_{\mathcal{A}_s}^\frac{1}{s}\texttt{tol}^{-\frac{1}{s}}.\label{approxfb}
\end{align}
\end{subequations}

The routine \textsf{APPROX} has been incorporated into the adaptive algorithms in e.g., \cite{CarPark2015,CarHella2017,CDR2019} because the $H(\divg)$-flux error estimators therein does not undergo the standard reduction property in Lemma \ref{continuity}. Nonlinear approximation routines for optimality of AFEMs are also used in e.g., \cite{BDD2004,Stevenson2007,CHX2009,CDN2012}.

Carstensen and Rabus \cite{CarHella2017} proposed a unified framework of optimality of AFEMs based on axioms of adaptivity with separate marking strategy, which might be useful for treating the general $f\in L^2\Lambda^n(\Omega)\backslash \mathcal{P}_r\Lambda^n(\mathcal{T}_0)$ in \eqref{nGHL} but is beyond the scope of this paper. Readers are also referred to, e.g., \cite{BM2008,CDN2012,CarPark2015} for other AFEMs with separate markings.
To be self-contained, {we shall follow the idea of \cite{CHX2009} and discuss an alternative adaptive approach to solving \eqref{nGHL}}. In fact, \textsf{AMFEM} could be initialized using a finer initial mesh $\CT_1$ and then applied to the modified problem:
Find $(\bar{\sigma},\bar{u})\in V^{n-1}\times V^n$, such that
\begin{equation}\label{mnGHL}
\begin{aligned}
\ab{\CC\bar{\sigma}}{\tau}-\ab{d\tau}{\bar{u}}&=0,\quad&&\tau\in V^{n-1},\\
\ab{d\bar{\sigma}}{v}&=\ab{Q_{\CT_1}f}{v},&& v\in V^n.
\end{aligned}
\end{equation}
For instance, $\CT_1$ could be the conforming completion of $\widetilde{\CT}_0.$
The resulting modified adaptive algorithm for \eqref{nGHL} is summarized as follows.

\begin{algorithm}
[\textsf{AMFEM--APPROX}] Input an initial mesh $\mathcal{T}_{0}$, a marking parameter $\theta\in(0,1)$, and an error tolerance $\texttt{tol}>0$.
\begin{itemize}
\item[]\textbf{Step 1:} Compute $\widetilde{\mathcal{T}}_0=\textsf{APPROX}(\mathcal{T}_0,f,\texttt{tol})$ and let $\mathcal{M}_0=\CT_0\backslash\widetilde{\mathcal{T}}_0.$ Then remove the hanging nodes in $\widetilde{\mathcal{T}}_0$ by \textsf{COMPLETE} to obtain $\CT_1$.
\item[]\textbf{Step 2:} Apply \textsf{AMFEM} to \eqref{mnGHL} using $\theta$ as the marking parameter  and $\CT_1$ obtained in \textbf{Step 1} as the initial mesh. The counter $\ell$ in \textsf{AMFEM} starts with $\ell=1.$
If $\sqrt{\mathcal{E}_\ell}\leq\texttt{tol}$ in the subroutine  \textsf{ESTIMATE} of \textsf{AMFEM}, then return  $(\sigma_L,u_L):=(\sigma_\ell,u_\ell)$, $\mathcal{T}_L:=\Tl,$ $L:=\ell$.
\end{itemize}
\end{algorithm}
If $k=n$, then $\osc_\ell(f)=0$ by the definition in Theorem \ref{eff}. Therefore, $|(\sigma,u,f)|_{\mathbb{A}_s}$ is independent of $f$, and we can use $|(\sigma,u)|_{\mathbb{A}_s}$ as an abbreviation for $|(\sigma,u,f)|_{\mathbb{A}_s}$ throughout the rest of the paper.
The rate of convergence of \textsf{AMFEM--APPROX} is described in the next corollary.
\begin{corollary}\label{qopt2}
Let $(\sigma_L,u_L)$ and $\mathcal{T}_L$ be the outputs of  \textsf{AMFEM--APPROX}. Assume $|(\bar{\sigma},\bar{u})|_{\mathbb{A}_{s}}<\infty$, $|f|_{\mathcal{A}_s}<\infty$ and $0<\theta<\theta_{*}$. In addition, assume \eqref{markcond}, \eqref{card}, and the conditions in Theorem \ref{eff} hold. Then there exists a constant $\overline{C}_{\emph{opt}}$ depending  only on $\theta, \theta_*, \CT_0, C_\CC, \Omega,$ and $s, \gamma, C_{\emph{conv}}$ such that
\begin{equation*}
\|\sigma-{\sigma}_L\|_{V_\CC}+\|u-u_L\|\leq \overline{C}_{\emph{opt}}\big(|(\bar{\sigma},\bar{u})|_{\mathbb{A}_s}+|f|_{\mathcal{A}_s}\big)\big(
\#\CT_L-\#\CT_0\big)^{-s}.
\end{equation*}
\end{corollary}
\begin{proof}
It follows from the continuous stability of \eqref{nGHL},  $\|f-Q_{\CT_1}f\|\leq\|f-Q_{\widetilde{\CT}_0}f\|$, and \eqref{approxfa} that
\begin{equation}\label{stab}
    \|\sigma-\bar{\sigma}\|_{V_\CC}+\|u-\bar{u}\|\lesssim\|f-Q_{\CT_1}f\|\leq\texttt{tol}.
\end{equation}
The continuous upper bound \eqref{ctsbd} ensures
\begin{equation}\label{tildebound}
    \|\bar{\sigma}-{\sigma}_L\|_{V_\CC}+\|\bar{u}-{u}_L\|\lesssim\sqrt{\mathcal{E}_L}\leq\texttt{tol}.
\end{equation}
Combining \eqref{stab}, \eqref{tildebound} with the triangle inequality yields
\begin{equation}\label{epsibound}
    \|{\sigma}-{\sigma}_L\|_{V_\CC}+\|{u}-{u}_L\|\lesssim\texttt{tol}.
\end{equation}

Now we assume $L\ge2$ otherwise the theorem is automatically true. The subroutine \textsf{AMFEM} in \textsf{AMFEM--APPROX} returns a sequence of estimators $\{\mathcal{E}_j\}^L_{j=1}$, marked sets $\{\mathcal{M}_j\}^{L-1}_{j=1}$, and meshes $\{\mathcal{T}_j\}_{j=2}^L$. A combination of \eqref{card},  $\#\mathcal{M}_0=\#\widetilde{\CT}_0-\#\CT_0$, and \eqref{approxfb} shows that
\begin{equation}\label{card2}
\#\CT_L-\#\CT_0\lesssim|f|^{\frac{1}{s}}_{\mathcal{A}_s}\texttt{tol}^{-\frac{1}{s}}+\sum_{j=1}^{L-1}\mathcal{M}_j.
\end{equation}
Since $\mathcal{E}_j$ is computed using the data $Q_{\CT_1}f\in\mathcal{P}_r\Lambda^n(\CT_1)$, the contraction in Theorem \ref{contraction} can be used to analyze the subroutine \textsf{AMFEM} based on the initial mesh $\CT_1$.
It follows from the cardinality estimate \eqref{card2},
Equation (5.9) of \cite{FFP2014}, $\sqrt{\mathcal{E}_{L-1}}>\texttt{tol}$ and $\CE_{L-1}\leq C_{\text{conv}}\gamma^{L-1-j}\CE_j$ with $j\leq L-1$ by Theorem \ref{contraction} that
\begin{equation}\label{card3}
\begin{aligned}
    &\#\CT_L-\#\CT_0\lesssim|f|^{\frac{1}{s}}_{\mathcal{A}_s}\texttt{tol}^{-\frac{1}{s}}+|(\bar{\sigma},\bar{u})|^{\frac{1}{s}}_{\mathbb{A}_s}\sum_{j=1}^{L-1}\sqrt{\mathcal{E}_j}^{-\frac{1}{s}}\\
    &\quad\leq|f|^{\frac{1}{s}}_{\mathcal{A}_s}\texttt{tol}^{-\frac{1}{s}}+|(\bar{\sigma},\bar{u})|^{\frac{1}{s}}_{\mathbb{A}_s}C_{\text{conv}}^\frac{1}{2s}\mathcal{E}_{L-1}^{-\frac{1}{2s}}\sum_{j=1}^{L-1}\gamma^{\frac{L-j-1}{2s}}\\
    &\quad\lesssim C_s\big(|f|^{\frac{1}{s}}_{\mathcal{A}_s}+|(\bar{\sigma},\bar{u})|^{\frac{1}{s}}_{\mathbb{A}_s}\big)\texttt{tol}^{-\frac{1}{s}},
\end{aligned}
\end{equation}
where $C_s=1+C_{\text{conv}}^\frac{1}{2s}(1-\gamma^{\frac{1}{2s}})^{-1}$. We conclude the proof from \eqref{epsibound} and \eqref{card3}.
\end{proof}
The rate of convergence in Corollary \ref{qopt2} relies on the membership $(\bar{\sigma},\bar{u})\in\mathbb{A}_s$ while $({\sigma},{u})\in\mathbb{A}_s$ is considered as the standard optimality condition in AFEM literature. It is well-known that the approximation class of Lagrange elements is closely related to Besov spaces, see, e.g., \cite{BDDP2002,Gaspoz2014}. However, practical criterion for determining the membership of approximation classes for mixed methods is not known in general.

On the other hand, the proof of optimality of the AMFEM for Poisson's equation in \cite{CarHella2017} requires that $(\sigma,u)\in\mathbb{A}_s$, that $\CT_0$ is sufficiently fine, and an extra small enough user-specified parameter for which the upper bound is difficult to compute.

\section{Applications}\label{secex}
In this section, we present several important applications of the results obtained in Sections \ref{secEstimate} and \ref{secqo}.
\subsection{Hodge Laplace equation}\label{subsecPoi}
For $1\leq i\leq n,$ let $$e_i=(-1)^{i+1}dx^1\wedge\cdots\wedge \widehat{dx^i}\wedge\cdots\wedge dx^n,$$
where $\widehat{dx^i}$ means that $dx^i$ is suppressed. There is a correspondence $j$ between $(n-1)$-forms and $\mathbb{R}^n$-valued functions:
\begin{align*}
\begin{CD}
    \sum_{i=1}^n\tau_ie_i@>j>>(\tau_1,\tau_2,\ldots,\tau_n).
\end{CD}
\end{align*}
On the other hand, an $n$-form can be identified with a scalar-valued function by \begin{align*}
    \begin{CD}
        vdx^1\wedge dx^2\wedge\cdots\wedge dx^n@>h>>v.
    \end{CD}
\end{align*}
Using the bijections $j, h$ and the definition \eqref{derivative}, $d^{n-1}$ is identified with the divergence operator $\divg$, the adjoint $\delta^n$ becomes the negative gradient $-\nabla$.

Given a vector space $\VB$, let $L^2(\Omega;\VB)$ denote the space of $\VB$-valued $L^2$-functions on $\Omega$. $V^{n-1}$ is isometric to $H(\divg;\Omega)=\{\tau\in L^2(\Omega;\mathbb{R}^n): \divg\tau\in L^2(\Omega)\}$ via $j$ and $V^n$ is isometric to $L^2(\Omega)$ via $h$. The Hodge Laplace equation \eqref{compactHL} with index $k=n$ reads
\begin{equation}\label{mixedPoisson}
    \begin{aligned}
        \sigma&=-\nabla u&&\text{ in }\Omega,\\ \divg\sigma&=f&&\text{ in }\Omega,\\
        u&=0&&\text{ on }\partial\Omega.
    \end{aligned}
\end{equation}
Since $\divg: H(\divg;\Omega)\rightarrow L^2(\Omega)$ is surjective, $\mathfrak{H}^n=\{0\}$ always vanishes on bounded Lipschitz domain $\Omega$.

Again using $j$ and $h,$
$\mathcal{P}_r\Lambda^n(\Tl)=\mathcal{P}^-_{r+1}\Lambda^n(\Tl)$ is identified with the space of piecewise polynomials of degree $\leq r$ without any continuity constraint,
\begin{align*}
    \mathcal{P}_{r+1}\Lambda^{n-1}(\Tl)&=\{\tau\in H(\divg;\Omega): \tau|_T\in\mathcal{P}_{r+1}(T;\mathbb{R}^n)\text{ for all }T\in\Tl\},\\
    \mathcal{P}^-_{r+1}\Lambda^{n-1}(\Tl)&=\{\tau\in H(\divg;\Omega): \tau|_T\in\mathcal{P}_r(T;\mathbb{R}^n)+x\mathcal{P}_r(T) \text{ for all }T\in\Tl\},
\end{align*}
where $\mathcal{P}_p(T;\mathbb{R}^n)$ is the space of $\mathbb{R}^n$-valued polynomials of degree $\leq p.$
The mixed method \eqref{DGHL}  with index $k=n$ and  $V_\ell^{n-1}\times V_\ell^n=\mathcal{P}^-_{r+1}\Lambda^{n-1}(\Tl)\times\mathcal{P}_r\Lambda^n(\Tl)$ or $V_\ell^{n-1}\times V_\ell^n=\mathcal{P}_{r+1}\Lambda^{n-1}(\Tl)\times\mathcal{P}_r\Lambda^n(\Tl)$ is indeed the RT or BDM element method, respectively. The $V^{n-1}$-norm is the $H(\divg)$-norm and $V^n$-norm is simply the $L^2$-norm. Under several assumptions, the rate of convergence of $(\sigma_\ell,u_\ell)\rightarrow(-\nabla u,u)$ in the $H(\divg)\times L^2$-norm is given by Corollary \ref{qopt2} or Theorem \ref{qopt}:
\begin{align*}
    \|\sigma-\sigma_\ell\|_{H(\divg)}+\|u-u_\ell\|\leq\overline{C}_{\text{opt}}\big(|(\bar{\sigma},\bar{u})|_{\mathbb{A}_s}+|f|_{\mathcal{A}_s}\big)\big(\#\CT_\ell-\#\CT_0\big)^{-s}.
\end{align*}

The identification of $V^k$ and $d^k, \delta^{k+1}$ with $k\leq n-2$ depends on the dimension of $\mathbb{R}^n.$ For example, the $(n-2)$-forms in $\mathbb{R}^2$ are automatically scalar-valued functions; the $(n-2)$-forms in $\mathbb{R}^3$ are realized by
\begin{equation*}
\begin{CD}
    v_1dx^1+v_2dx^2+v_3dx^3@>{s}>>(v_1,v_2,v_3).
\end{CD}
\end{equation*}
Using $s$ and $j,$ we have $d^{n-2}=\curl, \delta^{n-1}=\rot,$ where
\begin{equation*}
    \curl v=\left\{\begin{aligned}
    (\frac{\partial v}{\partial x_2 },-\frac{\partial v}{\partial x_1})\  (\text{ if } n=2)\\
    \nabla\times v\ (\text{ if }n=3)
\end{aligned}\right\},\quad \rot\tau=\left\{\begin{aligned}
    \frac{\partial\tau_2}{\partial x_1}-\frac{\partial\tau_1}{\partial x_2}\ (\text{ if }n=2)\\
    \nabla\times \tau\ (\text{ if }n=3)
\end{aligned}\right\}.
\end{equation*}
The space $ V^{n-2}$ is identified with
\begin{equation*}
    H(\curl;\Omega)=\left\{\begin{aligned}
    &H^1(\Omega),&&\quad n=2,\\
    &\{\phi\in L^2(\Omega;\mathbb{R}^3): \curl\phi\in L^2(\Omega;\mathbb{R}^3)\},&&\quad n=3.
\end{aligned}\right.
\end{equation*}
When $n=2$ or $3$, the $L^2$-de Rham complex \eqref{deRham} reduces to the well-known complexes
\begin{align*}
\begin{CD}
    H^1(\Omega)=H(\curl;\Omega)\xrightarrow{\curl}H(\divg;\Omega)\xrightarrow{\divg}L^2(\Omega)\text{ in }\mathbb{R}^2,\\
    H^1(\Omega)\xrightarrow{\nabla}H(\curl;\Omega)\xrightarrow{\curl}H(\divg;\Omega)\xrightarrow{\divg}L^2(\Omega)\text{ in }\mathbb{R}^3.
\end{CD}
\end{align*}
Given a face $S$ in $\Tl,$  $\tr|_S\star\tau=\tau_t$ is the tangential trace of $\tau$, where
\begin{align*}
    \tau_t=\tau\cdot t\text{ when }n=2,\quad\tau_t=\tau\times\nu\text{ when }n=3.
\end{align*}
Here $t$ and $\nu$ are unit tangent and normal to the face $S$, repectively.
The error indicator $\CE_\ell$ for Poisson's equation reads
\begin{align*}
    &\CE_{\ell}(T)=h_T^2\|\sigma_\ell+\nabla u_\ell\|_T^2+h_T\|\lr{u_\ell}\|^2_{\partial T}\\
    &\qquad+h_T^2\|\rot\sigma_\ell\|_T^2+h_T\|\lr{\sigma_{\ell,t}}\|^2_{\partial T}+\|f-\divg\sigma_\ell\|_T^2,
\end{align*}
which controls the error $\|\sigma-\sigma_\ell\|^2_{H(\divg)}+\|u-u_\ell\|^2$. Readers are referred to \cite{BV1996,CC1997} for other error estimators controlling the $H(\divg)\times L^2$-error.

We briefly describe other Hodge Laplace equations in $\mathbb{R}^2$ and $\mathbb{R}^3$.
The Hodge Laplacian problem \eqref{standardHL} with index $k\leq n-1$ and $n=2, 3$ reduces to the vector Laplacian problem:
\begin{equation}\label{interHL}
    \begin{aligned}
        \Delta u&=f\text{ in }\Omega,\\
        u\cdot t=0,~ \divg u&=0\text{ on }\partial\Omega\text{ when }k=2, n=3\text{ or }k=1, n=2,\\
        u\cdot\nu=0,~ \curl u\cdot t&=0\text{ on }\partial\Omega\text{ when }k=1, n=3.
    \end{aligned}
\end{equation}
In the case that $\mathfrak{H}^k=\{0\}$, Theorem \ref{qopt} confirms the quasi-optimal convergence rate of \textsf{AMFEM} for solving \eqref{interHL} based on the mixed formulation \eqref{GHL} with $\CC=\text{id}$ and finite element pairs \eqref{pairs}. In any dimension,  $\mathcal{P}_{r}^-\Lambda^0(\Tl)=\mathcal{P}_{r}\Lambda^0(\Tl)$ is the nodal finite element space of degree at most $r$. In $\mathbb{R}^3$,  $\mathcal{P}_{r+1}\Lambda^1(\Tl)$ and $\mathcal{P}^-_{r+1}\Lambda^1(\Tl)$ are called N\'ed\'elec edge finite element spaces \cite{Nedelec1980,Nedelec1986} in the classical context.
\subsection{Pseudostress-velocity formulation of the Stokes equation}
Given $\bm{f}\in L^2(\Omega;\mathbb{R}^n)$, the Stokes problem is to find $\bm{u}$ and $p$ with $\bm{u}|_{\partial\Omega}=0$, $\int_\Omega pdx=0$ satisfying
\begin{align*}
     -\Delta\bm{u}+\nabla p&=\bm{f}\text{ in }\Omega,\\
     \divg\bm{u}&=0\text{ in }\Omega,
\end{align*}
The pseudostress is $\bm{\sigma}=-\nabla\bm{u}+p\bm{I}_n$, where $\nabla$ denotes the row-wise gradient and $\bm{I}_n$ is the $n\times n$ identity matrix. The operator $\CC$ is given by
\begin{equation}\label{CC}
    \CC\bm{\sigma}:=\bm{\sigma}-\frac{1}{n}\text{Tr}(\bm{\sigma})\bm{I}_n,
\end{equation} where $\text{Tr}$ is the trace operator for square matrices. It is readily checked that $\CC$ is continuous and self-adjoint. Let $\Divg$ is the row-wise divergence for matrix-valued functions.
The Stokes problem is equivalent to the  pseudostress-velocity formulation (see, e.g., \cite{CaiWang2007,Cai2010b})
\begin{align*}
     \CC\bm{\sigma}&=-\nabla \bm{u}\text{ in }\Omega,\\ \Divg\bm{\sigma}&=\bm{f}\text{ in }\Omega,
\end{align*}
where $\bm{u}|_{\partial\Omega}=0$ and $\bm{\sigma}$ satisfies the compatibility condition $\int_\Omega\text{Tr}(\bm{\sigma})dx=0$. Let
$$\mathbf{V}^{n-1}=\{\bm{\tau}\in L^2(\Omega;\mathbb{R}^{n\times n}): \Divg\bm{\tau}\in L^2(\Omega;\mathbb{R}^n), \int_\Omega\text{Tr}(\bm{\tau})dx=0\}$$
and $\mathbf{V}^n=L^2(\Omega;\mathbb{R}^n)$.
The mixed variational formulation seeks find $\bm{\sigma}\in\mathbf{V}^{n-1}$ and $\bm{u}\in\mathbf{V}^n$ satisfying
\begin{equation}\label{STOKES}
    \begin{aligned}
     \ab{\CC\bm{\sigma}}{\bm{\tau}}-\ab{\Divg\bm{\tau}}{\bm{u}}&=0,\quad\bm{\tau}\in \mathbf{V}^{n-1},\\
     \ab{\Divg\bm{\sigma}}{\bm{v}}&=\ab{\bm{f}}{\bm{v}},\quad\bm{v}\in\mathbf{V}^n.
\end{aligned}
\end{equation}
It has been shown in \cite{ADG1984} that
$$\|\bm{\sigma}\|\lesssim\|\CC^\frac{1}{2}\bm{\sigma}\|+\|\Divg\bm{\sigma}\|_{H^{-1}(\Omega)}.$$
Combining it with $\|\Divg\bm{\sigma}\|_{H^{-1}(\Omega)}\leq\|\Divg\bm{\sigma}\|$ and the continuity of $\CC$, the assumption \eqref{equivCC} is verified.

For a vector- or matrix-valued function $\bm{v}$, let $\bm{v}_i$ denote the $i$-th entry or $i$-th row of $\bm{v}$, respectively. Let $\mathbf{V}_\ell^n=\{\bm{v}\in L^2(\Omega;\mathbb{R}^n): \bm{v}_i\in V_\ell^n, i=1,\ldots,n\}$ and $\mathbf{V}_\ell^{n-1}=\{\bm{\tau}\in L^2(\Omega;\mathbb{R}^{n\times n}): \bm{\tau}_i\in V_\ell^{n-1}, i=1,\ldots,n, \int_\Omega\text{Tr}(\bm{\tau})dx=0\},$ where $V_\ell^{n}$ and $V_\ell^{n-1}$ are given in  Subsection \ref{subsecPoi}. The mixed method for \eqref{STOKES} seeks
$\bm{\sigma}_\ell\in\mathbf{V}_\ell^{n-1}$ and $\bm{u}_\ell\in\mathbf{V}_\ell^n$ satisfying
\begin{equation}\label{DSTOKES}
    \begin{aligned}
     \ab{\CC\bm{\sigma}_\ell}{\bm{\tau}}-\ab{\Divg\bm{\tau}}{\bm{u}_\ell}&=0,\quad\bm{\tau}\in \mathbf{V}_\ell^{n-1},\\
     \ab{\Divg\bm{\sigma}_\ell}{\bm{v}}&=\ab{\bm{f}}{\bm{v}}\quad \bm{v}\in\mathbf{V}_\ell^n.
\end{aligned}
\end{equation}

Let $ L^2\Lambda^k(\Omega;\mathbb{R}^n)$ denote the space of all $\mathbb{R}^n$-valued $k$-forms $\omega$, namely,
$$\omega=\sum_{1\leq \alpha_{1}<\cdots<\alpha_{k}\leq n}\omega_{\alpha}dx^{\alpha_{1}}\wedge\cdots\wedge dx^{\alpha_{k}},$$
where each $\omega_\alpha\in L^{2}({\Omega;\mathbb{R}^n})$. The theory of de Rham complex in subsection \ref{secdeRham} can be directly extended to the vector-valued case.
There is a natural correspondence $\mathbf{j}$ between $\mathbb{R}^n$-valued $(n-1)$-forms and ${n\times n}$ matrix-valued functions:
\begin{align*}
\begin{CD}
    \sum_{i=1}^n\tau_ie_i@>\mathbf{j}>>(\tau_1,\tau_2,\ldots,\tau_n)^T.
\end{CD}
\end{align*}
Let $\mathbf{d}^{n-2}: L^2\Lambda^{n-2}(\Omega;\mathbb{R}^n)\rightarrow L^2\Lambda^{n-1}(\Omega;\mathbb{R}^n)$ denote the exterior derivative for vector-valued forms, $\text{D}=\mathbf{j}\circ \mathbf{d}^{n-2}$ and
\begin{equation*}
    \mathbf{V}^{n-2}=\{\bm{v}\in L^2\Lambda^{n-2}(\Omega;\mathbb{R}^n): \text{D}\bm{v}\in \mathbf{V}^{n-1}\}.
\end{equation*}
As in the scalar case, it is readily checked that the following is a closed Hilbert complex:
\begin{equation}\label{Scomplex}
\begin{CD}
    \mathbf{V}^{n-2}@>\text{D}>>\mathbf{V}^{n-1}@>\Divg>>\mathbf{V}^n.
\end{CD}
\end{equation}
Let $\mathbf{V}_\ell^{n-2}=\{v\in\mathbf{V}^{n-2}: v_i\in V_\ell^{n-2}, i=1,\ldots,n\}.$ The discrete subcomplex reads
\begin{equation}\label{DScomplex}
\begin{CD}
    \mathbf{V}_\ell^{n-2}@>\text{D}>>\mathbf{V}_\ell^{n-1}@>\Divg>>\mathbf{V}_\ell^n.
\end{CD}
\end{equation}
The surjectivity of $\Divg: \mathbf{V}^{n-1}\rightarrow\mathbf{V}^n$ implies the $n$-th cohomology group $\mathfrak{H}^n$ vanishes. The regular decomposition in Theorem \ref{rd} can be applied to each row of test functions in $\mathbf{V}^{n}$ and $\mathbf{V}^{n-1}$.

To apply the theory in Sections \ref{secEstimate} and \ref{secqo}, it suffices to construct a local $V$-bounded cochain projection ${\Pi}_\ell$ from \eqref{Scomplex} to \eqref{DScomplex} as well as an $L^2$-bounded cochain projection $\bar{\Pi}_\ell$ from
\begin{equation*}
    \begin{CD}
        L^2\Lambda^{n-2}(\Omega;\mathbb{R}^n)@>\text{D}>>L^2(\Omega;\mathbb{R}^{n\times n})@>\Divg>>L^2(\Omega;\mathbb{R}^{n})
    \end{CD}
\end{equation*}to \eqref{DScomplex}; compare with Theorem \ref{cochainproj} and Corollary \ref{disrd}. In fact, $\Pi_\ell$ and $\bar{\Pi}_\ell$ can be contructed from  $\pi_\ell$ and $\bar{\pi}_\ell$, respectively. As described in subsection \ref{subsecPoi}, $\pi^{n-1}_\ell$ and  $\pi^n_\ell$ can be applied to functions in $H(\divg;\Omega)$ and $L^2(\Omega)$, respectively. Let $\bm{\pi}_\ell^{n-1}$ and $\bm{\pi}_\ell^n$ denote the row-wise version of $\pi_\ell^{n-1}$ and $\pi_\ell^n$, respectively. Let ${\Pi}_\ell^n=\bm{\pi}_\ell^n: \mathbf{V}^n\rightarrow\mathbf{V}_\ell^n$ and
${\Pi}_\ell^{n-1}: \mathbf{V}^{n-1}\rightarrow\mathbf{V}_\ell^{n-1}$ be defined by
$$\Pi_\ell^{n-1}\bm{\tau}=\bm{\pi}_\ell^{n-1}\bm{\tau}-\frac{\bm{I}_n}{n|\Omega|}\int_\Omega\text{Tr}(\bm{\pi}_\ell^{n-1}\bm{\tau})dx$$
so that ${\Pi}_\ell^{n-1}(\mathbf{V}^{n-1})=\mathbf{V}_\ell^{n-1}.$ Note that $L^2\Lambda^{n-2}(\Omega;\mathbb{R}^n)=[L^2\Lambda^{n-2}(\Omega)]^n$, the Cartesian product of $L^2\Lambda^{n-2}(\Omega)$ with $n$ copies.
We can take $\bm{\pi}_\ell^{n-2}: L^2\Lambda^{n-2}(\Omega;\mathbb{R}^n)$ $\rightarrow[V_\ell^{n-2}]^n$ to be the component-wise version of $\pi_\ell^{n-2}$. For $\bm{w}\in\mathbf{V}^{n-2}$, let
$$\Pi_\ell^{n-2}\bm{w}:=\bm{\pi}_\ell^{n-2}\bm{w}-\frac{\bm{\mu}}{n(n-1)|\Omega|}\int_\Omega\text{Tr}(D\bm{\pi}_\ell^{n-2}\bm{w})dx,$$
where $\bm{\mu}=(\kappa e_1,\kappa e_2,\ldots,\kappa e_n)^T$ with $\kappa$ given in \eqref{kappa}. For example,
\begin{equation*}
    \bm{\mu}=\begin{pmatrix}x_2\\-x_1\end{pmatrix}\text{ when }n=2,\quad \bm{\mu}=\frac{1}{2}\begin{pmatrix}x_2dx^3-x_3dx^2\\x_3dx^1-x_1dx^3\\x_1dx^2-x_2dx^1\end{pmatrix}\text{ when }n=3.
\end{equation*}
Using the formula $(d\kappa+\kappa d)e_i=(n-1)e_i$ (cf.~Theorem 3.1 in \cite{AFW2006}) and $de_i=0$ with $1\leq i\leq n$, we have $\mathbf{d}^{n-2}\bm{\mu}=(n-1)(e_1,e_2,\ldots,e_n)^\intercal$ and thus $$\text{D}\bm{\mu}=(n-1)\bm{I}_{n},\quad\int_\Omega\text{Tr}(\text{D}\Pi_\ell^{n-2}\bm{w})dx=0.$$ Combining it with $\kappa e_i\in\mathcal{P}_1^-\Lambda^{n-2}(\Tl)\subseteq V_\ell^{n-2}$, we have $\Pi_\ell^{n-2}\bm{w}\in\mathbf{V}_\ell^{n-2}$.
Note that $\Pi_\ell$ is simply obtained by subtracting a global constant or differential form from $\bm{\pi}_\ell$. Therefore $\Pi_\ell$ is a local $V$-bounded cochain projection satisfying the properties in Theorem \ref{cochainproj}. The $L^2$-bounded projection $\bar{\Pi}_\ell$ can be constructed in the same way using $\bar{\pi}_\ell.$

Let
\begin{equation*}
\begin{aligned}
    &\CE_{\ell}(T)=h_T^2\|\CC\bm{\sigma}_\ell+\nabla \bm{u}_\ell\|_T^2+h_T\|\lr{\bm{u}_\ell}\|^2_{\partial T}\\
    &\qquad+h_T^2\|\text{D}^*\CC\bm{\sigma}_\ell\|_T^2+h_T\|\lr{\tr\star\mathbf{j}^{-1}(\bm{\sigma}_{\ell})}\|^2_{\partial T}+\|f-\Div\bm{\sigma}_\ell\|_T^2,
\end{aligned}
\end{equation*}
where $\text{D}^*$ is the adjoint of $\text{D}$. In $\mathbb{R}^2$ or $\mathbb{R}^3$,
$\text{D}^*=\Rot$, where $\Rot$ is the row-wise version of $\rot;$ ${\tr\star\mathbf{j}^{-1}(\bm{\sigma}_{\ell})}=\bm{\sigma}_{\ell,t}$, the row-wise tangential trace of $\bm{\sigma}_\ell.$
It follows from Theorems \ref{disupper} and \ref{eff} that $\CE_\ell$ is reliable and efficient for controlling $\|\CC^\frac{1}{2}(\bm{\sigma}-\bm{\sigma}_\ell)\|^2+\|\Divg(\bm{\sigma}-\bm{\sigma}_\ell)\|^2+\|\bm{u}-\bm{u}_\ell\|^2$. It follows from Corollary \ref{qopt2} that the rate of convergence of \textsf{AMFEM--APPROX} for \eqref{STOKES} is
\begin{align*}
    &\|\CC^\frac{1}{2}(\bm{\sigma}-\bm{\sigma}_\ell)\|+\|\Divg(\bm{\sigma}-\bm{\sigma}_\ell)\|+\|\bm{u}-\bm{u}_\ell\|\\
    &\qquad\leq\overline{C}_{\text{opt}}\big(|(\bar{\bm{\sigma}},\bar{\bm{u}})|_{\mathbb{A}_s}+|\bm{f}|_{\mathcal{A}_s}\big)\big(\#\CT_\ell-\#\CT_0\big)^{-s}.
\end{align*}
If $\bm{f}\in\bm{V}_0^n$, the right hand side of the previous estimate could be replaced by $\overline{C}_{\text{opt}}|(\bm{\sigma},\bm{u})|_{\mathbb{A}_s}\big(\#\CT_\ell-\#\CT_0\big)^{-s}$.

\section*{acknowledgements}
The author would like to thank the two anonymous referees for their critical and constructive comments that significantly improves the quality of this paper.

\providecommand{\bysame}{\leavevmode\hbox to3em{\hrulefill}\thinspace}
\providecommand{\MR}{\relax\ifhmode\unskip\space\fi MR }
\providecommand{\MRhref}[2]{%
  \href{http://www.ams.org/mathscinet-getitem?mr=#1}{#2}
}
\providecommand{\href}[2]{#2}

\end{document}